\documentclass[10pt]{amsart}

\usepackage{amssymb,verbatim}
\usepackage[colorlinks=true,urlcolor=blue, citecolor=red,linkcolor=blue,linktocpage,pdfpagelabels, bookmarksnumbered,bookmarksopen]{hyperref}
\usepackage[hyperpageref]{backref}

\newtheorem{lm}{Lemma}[section]
\newtheorem{prop}[lm]{Proposition}
\newtheorem{coro}[lm]{Corollary}
\newtheorem{teo}[lm]{Theorem}
\theoremstyle{definition}
\newtheorem{oss}[lm]{Remark}
\newtheorem{defi}[lm]{Definition}
\newtheorem*{ack}{Acknowledgements}

\numberwithin{equation}{section}

\author[Brasco]{Lorenzo Brasco}
\address[L.\ Brasco]{Dipartimento di Matematica e Informatica
\newline\indent
Universit\`a degli Studi di Ferrara
\newline\indent
Via Machiavelli 35, 44121 Ferrara, Italy}
\address{{\it and }
Aix-Marseille Universit\'e, CNRS
\newline\indent
Centrale Marseille, I2M, UMR 7373, 39 Rue Fr\'ed\'eric Joliot Curie
\newline\indent
13453 Marseille, France}
\email{lorenzo.brasco@unife.it}

\author[Lindgren]{Erik Lindgren}
\address[E. Lindgren]{Department of Mathematics, Royal Institute of Technology
\newline\indent
10044 Stockholm, Sweden}
\email{eriklin@kth.se}

\date{\today}


\title[Higher Sobolev regularity]{Higher Sobolev regularity\\ for the fractional $p-$Laplace equation\\ in the superquadratic case}

\begin{document}

\begin{abstract}
We prove that for $p\ge 2$, solutions of equations modeled by the fractional $p-$Laplacian improve their regularity on the scale of fractional Sobolev spaces. Moreover, under certain precise conditions, they are in $W^{1,p}_{loc}$ and their gradients are in a fractional Sobolev space as well. The relevant estimates are stable as the fractional order of differentiation $s$ reaches $1$.
\end{abstract}

\maketitle
\begin{center}
\begin{minipage}{10cm}
\small
\tableofcontents
\end{minipage}
\end{center}

\section{Introduction}

\subsection{Aim of the paper}
Let $2\le p<\infty$, $\Omega\subset\mathbb{R}^N$ be an open set and consider a local weak solution $u$ of the $p-$Laplace equation
\[
-\Delta_p u=0,\qquad \mbox{ in }\Omega.
\]
This means that $u\in W^{1,p}_{loc}(\Omega)$ and verifies
\[
\int_{\Omega'} \langle |\nabla u|^{p-2}\,\nabla u,\nabla \varphi\rangle\,dx=0,
\]
for every open set $\Omega'$ compactly contained in $\Omega$ and every $\varphi\in W^{1,p}_0(\Omega')$. Thus the operator $-\Delta_p$ arises from the first variation of the $W^{1,p}$ Sobolev seminorm. A classical regularity result by Uhlenbeck asserts that (see \cite[Lemma 3.1]{Uh})
\[
|\nabla u|^\frac{p-2}{2}\,\nabla u \in W^{1,2}_{loc}(\Omega).
\]
This in turn implies the following higher differentiability for the gradient itself
\begin{equation}
\label{classic}
\nabla u\in W^{\tau,p}_{loc}(\Omega),\qquad \mbox{ for every }0<\tau<\frac{2}{p},
\end{equation}
see also \cite[Proposition 3.1]{Mi03} for a more comprehensive result. 
\par
In this paper we want to tackle this regularity issue for weak solutions of nonlocal and nonlinear equations like the {\it fractional $p-$Laplace equation} 
\begin{equation}
\label{local}
(-\Delta_p)^s u=0,
\end{equation}
and prove the analogue of \eqref{classic}. Here $0<s<1$ is given. In order to clarify the content of this paper, it is useful to recall that various different definitions of fractional (or nonlocal) $p-$Laplacian have been recently proposed (see for example \cite{BCF}, \cite{CJ} and \cite{SV}). The definition considered in this paper is the variational one. That is, if for every open set $E\subset\mathbb{R}^N$ we define the $W^{s,p}$ Gagliardo seminorm
\[
[u]_{W^{s,p}(E)}:=\left(\int_{E}\int_{E} \frac{|u(x)-u(x)|^p}{|x-y|^{N+s\,p}}\,dx\,dy\right)^\frac{1}{p},
\]
then the operator $(-\Delta_p)^s$ arises as the first variation of
\[
u\mapsto [u]^p_{W^{s,p}(\mathbb{R}^N)}.
\] 
This is in analogy with the case of $-\Delta_p$, which formally corresponds to the case $s=1$. Operators of this type were, to the best of our knowledge, first considered in \cite{AMRT} and \cite{IM}. 
A weak solution $u$ of \eqref{local} verifies
\[
\int_{\mathbb{R}^N}\int_{\mathbb{R}^N} \frac{|u(x)-u(x)|^{p-2}\,(u(x)-u(y))}{|x-y|^{N+s\,p}}\,\Big(\varphi(x)-\varphi(y)\Big)\,dx\,dy=0,
\]
for every $\varphi\in W^{s,p}$ with compact support. The reader worried about the sloppiness of this definition is invited to jump to Definition \ref{defi:localweak} below. There one may find the precise description of the equation and the definition of a weak solution. 
\par
We point out that for ease of readability for the moment we just focus on the operator $(-\Delta_p)^s$. But indeed we will treat more general operators, where the singular kernel $(x,y)\mapsto |x-y|^{-N-s\,p}$ is replaced by some slight generalizations.
\vskip.2cm
Very recently the operator $(-\Delta_p)^s$ has been much studied and the low regularity of solutions is now quite well understood. The first important paper on the subject is \cite{DKP} by Di Castro, Kuusi and Palatucci. There local H\"older regularity for solutions of \eqref{local} is proved by building De Giorgi-type techniques for the nonlocal and nonlinear setting, in a
similar spirit as it was first done for the case $p=2$ by Kassman in
\cite{Ka}. In the companion paper \cite{DKP2}, the same authors also proved the Harnack inequality for solutions of the homogeneous equation. As for the inhomogeneous equation
\begin{equation}
\label{localF}
(-\Delta_p)^s u=f,
\end{equation}
it is unavoidable to mention the impressive paper \cite{KMS1} by Kuusi, Mingione and Sire, where very refined pointwise estimates of potential type are proved. These lead for example to local continuity of the solution under sharp assumptions on $f$ (see \cite[Corollary 1.2]{KMS1}). It is worth mentioning that \cite{KMS1} considers a general measure datum $f$, not necessarily belonging to the natural dual Sobolev space. In this case, the concept of solution has to be carefully defined.
Finally, Iannizzotto, Mosconi and Squassina in \cite{IMS} (see also \cite{IMS2}) succeeded in proving global H\"older regularity for the solution of the Dirichlet problem
\[
\left\{\begin{array}{rccl}
(-\Delta_p)^s u&=&f,& \mbox{ in }\Omega,\\
u&=& 0,& \mbox{ in }\mathbb{R}^N\setminus \Omega,
\end{array}
\right.
\]
under appropriate assumptions on the data $f$ and $\Omega$ (see \cite[Theorem 1.1]{IMS}).
\par
On the contrary, as for higher differentiability of solutions, the picture is less clear. Some results on this subject are contained in the recent papers \cite{Co,KMS2} and \cite{Sc} (see also \cite{KMS3}). We postpone comments on these papers, let us now proceed to present our main result.

\subsection{Some expedient definitions}
In order to neatly state our contribution, we need some definitions. We start with a couple of weighted spaces.
\begin{defi}[Special spaces]
\label{defi:snails}
Let $1<p<\infty$ and $0<s<1$. We introduce the weighted Lebesgue space
\[
\mathcal{X}^{p}_s=\left\{\psi\in L^p_{loc}(\mathbb{R}^N)\, :\, \int_{\mathbb{R}^N} \frac{|\psi(x)|^p}{(1+|x|)^{N+s\,p}}\,dx<+\infty\right\}.
\]
For $0\le t\le 1$, we also consider the weighted Nikol'skii-type space
\[
\mathcal{Y}^{t,p}_s=\left\{\psi\in\mathcal{X}^p_s\, :\, \exists h_0>0 \mbox{ such that }\sup_{0<|h|<h_0} \int_{\mathbb{R}^N} \left|\frac{\psi(x+h)-\psi(x)}{|h|^t}\right|^p\, \frac{dx}{(1+|x|)^{N+s\,p}}<+\infty\right\}.
\]
It is intended that $|h|^0=1$, so that for $t=0$ it is easy to see that $\mathcal{Y}^{0,p}_{s}=\mathcal{X}^{p}_{s}$.
\end{defi}
We need to introduce a nonlocal quantity which is very similar to that of {\it nonlocal tail} of a function, introduced in \cite{DKP}. Since the two definitions differ slightly, we prefer to introduce a different notation and terminology. In what follows, the writing $F\Subset E$ means that both $F$ and $E$ are open sets of $\mathbb{R}^N$, such that the closure of $F$ is a compact set contained in $E$.
\begin{defi}[Snail norms]
Let $1<p<\infty$, $0<s<1$ and $\psi\in \mathcal{X}^p_s$. For every open and bounded set $E\subset\mathbb{R}^N$, we set
\[
\mathrm{Snail}(\psi;x,E):=\left[|E|^\frac{s\,p}{N}\,\int_{\mathbb{R}^N\setminus E}\frac{|\psi(y)|^{p}}{|x-y|^{N+s\,p}}\,dy\right]^\frac{1}{p},\qquad x\in E.
\]
Then for every $F\Subset E\Subset \mathbb{R}^N$ the following quantity is well-defined
\[
\langle\psi\rangle_{\mathcal{X}^{p}_s(F;E)}:=\left(\int_{E} |\psi|^p\,dx+\int_{F} \mathrm{Snail}(\psi;x,E)^{p}\,dx\right)^\frac{1}{p}.
\]
If $\psi\in \mathcal{Y}^{t,p}_s$ for some $0\le t\le 1$, we can also define the Nikol'skii-type quantity
\begin{equation}
\label{snails}
\begin{split}
\langle\psi\rangle_{\mathcal{Y}^{t,p}_s(F;E)}&:=\sup_{0<|h|<\frac{1}{2}\,d(F,E)}\left(\int_{F} \mathrm{Snail}\left(\frac{\psi(\cdot+h)-\psi}{|h|^t};x,E\right)^{p}\,dx\right)^\frac{1}{p},
\end{split}
\end{equation}
where we set $d(F,E):=\mathrm{dist}(F,\mathbb{R}^N\setminus E)$.
\end{defi} 
\begin{defi}[Operator and local weak solutions]
\label{defi:localweak}
Let $1<p<\infty$ and $0<s<1$. We consider a measurable function $K:\mathbb{R}^N\to [0,+\infty)$ satisfying
\begin{equation}\label{Kassummption}
\frac{1}{\Lambda}\,|z|^{N+s\,p}\le K(z) \le\Lambda\,|z|^{N+s\,p}\quad \mbox{for all $z\in \mathbb{R}^N$},
\end{equation}
for some $\Lambda\ge 1$. Let $\Omega\subset\mathbb{R}^N$ be an open set, given $f\in L^{p'}_{loc}(\Omega)$, we say that $u\in W^{s,p}_{loc}(\Omega)\cap \mathcal{X}^p_s$ is a {\it local weak solution} of 
\begin{equation}
\label{localK}
(-\Delta_{p,K})^s u = f\quad \mbox{ in }\Omega,
\end{equation}
if
\begin{equation}
\label{weak}
\int_{\mathbb{R}^N} \int_{\mathbb{R}^N} \frac{|u(x)-u(y)|^{p-2}\,(u(x)-u(y))}{K(x-y)}\,(\varphi(x)-\varphi(y))\,dx\,dy=\int_{\Omega'} f\,\varphi\,dx,
\end{equation}
for every $\Omega'\Subset \Omega$ and every $\varphi\in W^{s,p}(\Omega)$ such that $\varphi\equiv 0$ in $\Omega\setminus \Omega'$. It is intended that the test functions $\varphi$ are extended by $0$ outside $\Omega$ in \eqref{weak}. The assumptions on $u$ and $K$ guarantee that the double integral in the left-hand side of \eqref{weak} is absolutely convergent.
\par 
In the case $K(z)=|z|^{N+s\,p}$, we will simply write $(-\Delta_p)^s$ in place of $(-\Delta_{p,K})^s$.
\end{defi}
\begin{oss}
As for the $p-$Laplacian, we do not assume local solutions to belong to $W^{s,p}(\Omega)$, but only to $W^{s,p}_{loc}(\Omega)$. For this reason, in order to give sense to \eqref{weak}, we require test functions to vanish identically outside compactly contained subsets of $\Omega$.
\end{oss}
\subsection{Main results}
The following is our main result. The parameter $t$ below measures the degree of differentiability of the solution ``at infinity''. The value $t=0$ is admitted as well, thus the differentiability ``at infinity'' is not necessary to improve the local one. The case of the $p-$Laplacian formally corresponds to taking $s=t=1$ in \eqref{spiritualhealing} below. In this case, the result boils down to the aforementioned one \eqref{classic}.
\begin{teo}[Higher differentiability]
\label{teo:high}
Let $p\ge 2$, $0<s<1$ and $0\le t\le s$. Let $u\in W^{s,p}_{loc}(\Omega)\cap \mathcal{Y}^{t,p}_s$ be a local weak solution of 
\[
(-\Delta_{p,K})^su= f,\qquad \mbox{ in }\Omega, 
\]
with $f\in W^{s,p'}_{loc}(\Omega)$ and $K$ verifying \eqref{Kassummption}. 
For every ball $B_R\Subset\Omega$ we define
\begin{equation}
\label{AR}
\begin{split}
\mathcal{A}_R(u,f)&:=\left(R^{s\,p}\,\left[u\right]^p_{W^{s,p}(B_R)}+\frac{1}{s\,(1-s)}\,\|u\|^p_{L^p(B_{R})}\right)\\
&+\frac{1}{s}\langle u\rangle^p_{\mathcal{X}^{p}_s(B_{\frac{3}{4}\,R},B_R)}+R^{t\,p}\,\langle u\rangle^p_{\mathcal{Y}^{t,p}_s(B_{\frac{3}{4}\,R};B_{\frac{7}{8}\,R})}\\
&+R^{s\,p\,p'}\,\left(R^{s\,p'}\,\big[(1-s)\,f\big]_{W^{s,p'}(B_R)}^{p'}+\frac{1}{s\,(1-s)}\,\big\|(1-s)\,f\big\|_{L^{p'}(B_R)}^{p'}\right).
\end{split}
\end{equation}
Then we have:
\begin{itemize}
\item[{\it i)}] if $\boxed{t+s\,p\le p-1}$
\[
u\in W^{\tau,p}_{loc}(\Omega),\qquad \mbox{ for every } s\le\tau<\frac{t+s\,p}{p-1},
\]
and for every ball $B_R\Subset \Omega$ there holds the scaling invariant estimate
\begin{equation}
\label{screambloodygore}
\begin{split}
[u]^p_{W^{\tau,p}(B_{R/2})}&\le \frac{\mathcal{C}_1}{R^{\tau\,p}}\,\mathcal{A}_R(u,f),
\end{split}
\end{equation}
for some $\mathcal{C}_1=\mathcal{C}_1(N,p,s,\Lambda,t,\tau)>0$;
\vskip.2cm
\item[{\it ii)}] if $\boxed{t+s\,p>p-1}$ we set
\[
\Gamma:=\frac{1+t+s\,p}{p},
\]
then
\[
u\in W^{1,p}_{loc}(\Omega)\qquad \mbox{ and }\qquad \nabla u\in W^{\tau,p}_{loc}(\Omega),\quad\mbox{ for every } \tau<\Gamma-1,
\]
and for every ball $B_R\Subset \Omega$ there hold the scaling invariant estimates
\begin{equation}
\label{leprosy}
\begin{split}
\|\nabla u\|^p_{L^p(B_{R/2})}&\le \frac{\mathcal{C}_2}{R^p}\,\mathcal{A}_R(u,f),
\end{split}
\end{equation}
and
\begin{equation}
\label{spiritualhealing}
\begin{split}
[\nabla u]^p_{W^{\tau,p}(B_{R/4})}&\le \frac{(2-\Gamma)^{-p}\,(\Gamma-1)^{-p}}{(\Gamma-1-\tau)\,\tau\,}\,\frac{\mathcal{C}_3}{R^{p\,(1+\tau)}}\,\mathcal{A}_R(u,f).
\end{split}
\end{equation}
for some $\mathcal{C}_2=\mathcal{C}_2(N,p,s,\Lambda,t)>0$ and $\mathcal{C}_3=\mathcal{C}_3(N,p,s,\Lambda,t)>0$.
\end{itemize}
\end{teo}
\begin{oss}[Behaviour of the constants]
\label{oss:constants}
In the second case, if the crucial quantity $t+s\,p$ is sufficiently well-detached from $p-1$, then it is possible to make explicit the dependence of $\mathcal{C}_2$ and $\mathcal{C}_3$ on $s$.
More precisely, let us fix $\ell_0>p$, then for every $0\le t\le s<1$ such that 
\[
t+s\,(p+1)\ge \ell_0,
\]
estimates \eqref{leprosy} and \eqref{spiritualhealing} can be replaced by
\begin{equation}
\label{human}
\begin{split}
\|\nabla u\|^p_{L^p(B_{R/2})}&\le \frac{C}{(\ell_0-p)^p}\,(1-s)\,\frac{\mathcal{A}_R(u,f)}{R^p},
\end{split}
\end{equation}
and
\begin{equation}
\label{individual}
\begin{split}
[\nabla u]^p_{W^{\tau,p}(B_{R/4})}&\le \frac{(2-\Gamma)^{-p}\,(\Gamma-1)^{-p}}{(\Gamma-1-\tau)\,\tau\,}\,\frac{C}{(\ell_0-p)^p}\,(1-s)\,\frac{\mathcal{A}_R(u,f)}{R^{p\,(1+\tau)}},
\end{split}
\end{equation}
with $C>0$ depending on $N, p$ and $\Lambda$ only. We will come back on the relevance of these estimates in a while.
\end{oss}
\begin{oss}[H\"older continuity via embedding]
By using Morrey-type embeddings for fractional Sobolev spaces (see \cite[Theorem 7.57]{Ad}), we get that a local weak solution $u \in  W^{s,p}_{loc}(\Omega)\cap \mathcal{Y}^{t,p}_s$ is locally H\"older continuous for $p\ge 2$, $0<s<1$ and $0\le t\le s$ such that
\[
t+s\,p>\frac{p-1}{p}\,N,\qquad \mbox{ if }\ t+s\,p\le p-1,
\]
or
\[
t+s\,p>N-1,\qquad \mbox{ if }\ t+s\,p>p-1.
\]
For example, in dimension $N=2$ this is always the case if $p\ge 2$ and $s>(p-1)/p$.   
\end{oss}
Before proceeding further, let us illustrate some particular cases of the previous result.
We start with the case where our solution $u$ is a priori known to be globally bounded, a situation that is quite natural if $u$ is constructed through viscosity methods (see \cite{Li}). 
\begin{coro}[Bounded solutions] 
Let $p\ge 2$ and $0<s<1$. Let $u\in W^{s,p}_{loc}(\Omega)\cap L^\infty(\mathbb{R}^N)$ be a local weak solution of \eqref{localK}, with $f\in W^{s,p'}_{loc}(\Omega)$ and $K$ verifying \eqref{Kassummption}.  Then 
\begin{itemize}
\item[{\it i)}] if $\boxed{s\le (p-1)/p}$
\[
u\in W^{\tau,p}_{loc}(\Omega),\qquad \mbox{ for every } \tau<\frac{s\,p}{p-1},
\]
\item[{\it ii)}] if $\boxed{s>(p-1)/p}$
\[
u\in W^{1,p}_{loc}(\Omega)\qquad \mbox{ and }\qquad \nabla u\in W^{\tau,p}_{loc}(\Omega),\quad\mbox{ for every } \tau<s-\frac{p-1}{p}.
\]
\end{itemize}
\end{coro}
\begin{proof}
The result follows from the simple observation that 
\[
L^\infty(\mathbb{R}^N)\subset \mathcal{X}^p_s=\mathcal{Y}^{0,p}_s,
\]
see \eqref{Lp} below. Thus we can apply Theorem \ref{teo:high} with $t=0$.
\end{proof}
An important case is that of nonlocal Dirichlet boundary value problems for the operator $(-\Delta_p)^s$. Indeed, since the ``boundary datum'' $g$ is imposed on the whole complement $\mathbb{R}^N\setminus \Omega$, the solution $u$ naturally inherits differentiability properties ``at infinity'' from $g$. We can tune the parameter $t$ accordingly and improve the result. As in \cite{BLP}, we use the notation $\widetilde W^{s,p}_0(\Omega)$ to denote the completion of $C_0^\infty(\Omega)$ with respect to the norm 
\[
\psi\mapsto \|\psi\|_{L^p(\Omega)}+[\psi]_{W^{s,p}(\mathbb{R}^N)}.
\] 
\begin{coro}[Dirichlet problems] 
Let $p\ge 2$ and $0<s<1$. Let $\Omega\subset\mathbb{R}^N$ be an open and bounded set. Given $f\in W^{s,p'}(\Omega)$, $g\in W^{s,p}(\mathbb{R}^N)$ and $K$ verifying \eqref{Kassummption}, we consider the (unique) solution $u\in W^{s,p}(\mathbb{R}^N)$ of the problem
\[
\left\{\begin{array}{rccl}
(-\Delta_{p,K})^s u&=&f,& \mbox{ in }\Omega,\\
u&=& g,& \mbox{ in }\mathbb{R}^N\setminus \Omega.
\end{array}
\right.
\]
This means that $u$ coincides with $g$ in $\mathbb{R}^N\setminus \Omega$ and verifies \eqref{weak} for every test function $\varphi\in \widetilde W^{s,p}_0(\Omega)$.
Then we have:
\begin{itemize}
\item[{\it i)}] if $\boxed{s\le (p-1)/(p+1)}$
\[
u\in W^{\tau,p}_{loc}(\Omega),\qquad \mbox{ for every } s\le\tau<s\,\frac{p+1}{p-1};
\]
\item[{\it ii)}] if $\boxed{s>(p-1)/(p+1)}$
\[
u\in W^{1,p}_{loc}(\Omega)\qquad \mbox{ and }\qquad \nabla u\in W^{\tau,p}_{loc}(\Omega),\quad\mbox{ for every } \tau<s\,\frac{p+1}{p}-\frac{p-1}{p}.
\]
\end{itemize}
\end{coro}
\begin{proof}
It is sufficient to observe that
\[
W^{s,p}(\mathbb{R}^N)\subset \mathcal{Y}^{s,p}_s,
\]
see \eqref{Wsp} below. Thus we can apply Theorem \ref{teo:high} with $t=s$.
\end{proof}
\subsection{Comments} Some comments are in order, we start with some words on the proof of Theorem \ref{teo:high}.
\begin{itemize}
\item{({\it About the proof\,})} The starting point of the proof of Theorem \ref{teo:high} is standard, we differentiate equation \eqref{weak} in a discrete sense. Then by testing the equation against fractional derivatives of the solution, i.e. quantities like
\begin{equation}
\label{fd}
\frac{u(x+h)-u(x)}{|h|^\vartheta},
\end{equation}
we establish a Caccioppoli-type inequality for finite differences of the solution (see Proposition \ref{prop:derivatives}). For the $p-$Laplacian this is a ``one shot'' estimate, i.e. by taking $\vartheta$ to be the exponent dictated by the hypothesis $u\in W^{1,p}_{loc}$ we directly reach \eqref{classic} from this Caccioppoli-type inequality. On the contrary, in the nonlocal case this estimate may in general be iterated. The number of possible iterations depends of course on $s$, namely on how close it is to $1$. Then the initial information $u\in W^{s,p}_
{loc}$ can be recursively improved. At each step the differentiability gain is on a ``hybrid scale'', which mixes two different ways of measuring fractional derivatives.
Roughly speaking, at each step we are estimating the $W^{s,p}$ seminorm (i.\,e.\ $s$ derivatives on the Gagliardo scale) of a finite difference \eqref{fd} (i.\,e.\ $\vartheta$ derivatives on the Nikol'skii scale). The main point of the iteration is to identify the resulting quantity as the norm of $s+\vartheta$ derivatives of the solution, measured again on the Nikol'skii scale. We point out that this is a genuine Besov-type estimate (see Lemma \ref{lm:triebel}). 
\vskip.2cm
\item{({\it The right-hand side\,})} As for the right-hand side $f$, the hypothesis $W^{s,p'}_{loc}$ is certainly too strong and we could improve the differentiability of the solution under weaker assumptions.  
On the other hand, we prefer to avoid further complications in the statement (and the proof) of Theorem \ref{teo:high}, thus for the moment we do not try to relaxe it.
\par
It is natural to expect that a suitable variant of Theorem \ref{teo:high} holds true also for very weak solutions with measure data, by using perturbative and approximation arguments as in \cite[Section 6]{Mi07}.
\vskip.2cm
\item{({\it Previous results\,})} Let us now make some comments on the aforementioned papers \cite{Co,KMS2,KMS3} and \cite{Sc}. Let us start with the linear case, corresponding to the choice $p=2$. In \cite{KMS2} and \cite{KMS3}, the authors consider general linear elliptic nonlocal equations like
\begin{equation}
\label{KMS}
\int_{\mathbb{R}^N}\int_{\mathbb{R}^N} \frac{(u(x)-u(y))}{\mathcal{K}(x,y)}\,\Big(\varphi(x)-\varphi(y)\Big)\,dx\,dy=\int f\, \varphi,\qquad \mbox{ for every }\varphi,
\end{equation}
where
\[
f\in L^{\frac{2\,N}{N+2\,s}+\delta_0}\qquad \mbox{ and }\qquad \frac{1}{\Lambda}\,|x-y|^{N+2\,s}\le \mathcal{K}(x,y) \le\Lambda\,|x-y|^{N+2\,s}. 
\]
They prove that a solution $u\in W^{s,2}(\mathbb{R}^N)$ is indeed in $W^{s+\delta,2+\delta}(\mathbb{R}^N)$ for some 
\[
0<\delta=\delta(N,s,\delta_0,\Lambda)<1-s,
\]
see \cite[Theorem 1.1]{KMS2}. The result is essentially not comparable with ours: it is weaker, but obtained under very general assumptions by using a different technique, i.e. a suitable {\it fractional Gehring Lemma} (see \cite[Theorem 1.2]{KMS2}).  
We just notice that as a consequence of Theorem \ref{teo:high}, in our case as well we can improve both the differentiability and the integrability exponent, just by using a standard interpolation argument.
\par
In \cite{Co} it is still considered the equation \eqref{KMS}, under the additional assumptions 
\[
f\in L^{2}\qquad \mbox{ and }\qquad \left|\frac{1}{\mathcal{K}(x+h,y+h)}-\frac{1}{\mathcal{K}(x,y)}\right| \le C\,\frac{|h|^{s}}{|x-y|^{N+2\,s}}. 
\]
Observe that the previous condition on $\mathcal{K}$ covers for example the case of kernels of the type $\mathcal{K}(x,y)=K(x-y)$. Then \cite[Theorem 2.2]{Co} shows that the solution gains ``almost'' $s-$derivatives, i.e. $u\in W^{2\,s-\tau,2}_{loc}$ for every $\tau>0$. The proof relies on differentiating twice the equation in discrete sense. Though limited to linear equations, we may notice that the result of \cite{Co} is stronger than our Theorem \ref{teo:high} in the case $p=2$. Indeed, if we consider Theorem \ref{teo:high} for $p=2$ and $s>1/2$ and we do not assume differentiability ``at infinity'' of the solution (as in \cite{Co}), i.e. we take $t=0$,  then we obtain $u\in W^{s+1/2-\tau,2}_{loc}$, for every $\tau>0$. We point out that this mismatch is not linked to the presence of the right-hand side $f$.
\par
As for the general case $p\ge 2$, in \cite{Sc} the author considers a ``regional'' version of \eqref{localF}, i.e. the equation
\[
\int_{\Omega}\int_{\Omega} \frac{|u(x)-u(y)|^{p-2}\,(u(x)-u(y))}{|x-y|^{N+s\,p}}\,\Big(\varphi(x)-\varphi(y)\Big)\,dx\,dy=\int_\Omega  f\, \varphi,\,dx\qquad \mbox{ for every }\varphi,
\]
where $f$ belongs to the dual space of $W^{s-\varepsilon(p-1),p}(\Omega)$, for some $\varepsilon>0$. 
In \cite[Theorem 1.3]{Sc} it is proved that there exists $\varepsilon_0=\varepsilon_0(s,p,\Omega)>0$ such that for $\varepsilon<\varepsilon_0$, a solution $u\in W^{s,p}(\Omega)$ is indeed in $W^{s+\varepsilon,p}_{loc}(\Omega)$.
\vskip.2cm
\item{({\it Limit as $s\nearrow 1$\,})} 
Finally, we conclude this list of comments by stressing that estimates \eqref{human} and \eqref{individual} display the correct dependence on the parameter $s$, at least in the asymptotical regime $s\nearrow 1$. Indeed, we recall that for a function $u\in W^{1,p}_{loc}$ we have the pointwise convergence
\[
\lim_{s\nearrow 1}(1-s)\,\int_{B_R}\int_{B_R} \frac{|u(x)-u(x)|^p}{|x-y|^{N+s\,p}}\,dx\,dy=C_{N,p}\,\int_{B_R} |\nabla u|^p\,dx,
\]
see \cite{BBM2,BBM1}. Moreover, we also have the $\Gamma-$convergence of the two functionals displayed above with respect to the strong $L^p$ topology, see \cite{BPS} and \cite{Po}. Thus, in the standard case $K(z)=|z|^{N+s\,p}$, the estimates of Theorem \ref{teo:high} can be used to prove that solutions of the fractional $p-$Laplace equation converge strongly in $W^{1,p}_{loc}$ to solutions of the usual $p-$Laplace equation as $s\nearrow 1$, under suitable assumptions.
For example, let $\Omega\subset\mathbb{R}^N$ be an open and bounded set and let $u_s$ be the unique solution of
\[
(-\Delta_p)^s u_s=f_s:=\frac{f}{1-s},\quad \mbox{ in }\Omega,\quad\qquad u=0,\quad \mbox{ in }\mathbb{R}^N\setminus\Omega.
\]
By using \eqref{human} and \eqref{individual} it is possible to show that $u_s$ converges strongly in $L^p(\Omega)\cap W^{1,p}_{loc}(\Omega)$ to the unique solution of
\[
-\Delta_p u=f,\quad \mbox{ in }\Omega,\quad\qquad u=0,\quad \mbox{ on }\partial\Omega.
\]

\end{itemize}
 
\subsection{Plan of the paper}
In Section \ref{sec:preliminaries} we introduce all the definitions and the functional analytic stuff that will be needed throughout the whole paper. The core of the paper is Section \ref{sec:evil}, where the fundamental estimates are settled down. These are the Caccioppoli-type inequality of Proposition \ref{prop:derivatives} and the Besov-Nikol'skii differentiability improvement of Lemma \ref{lm:triebel}. The proof of Theorem \ref{teo:high} is then contained in Section \ref{sec:proof}. In the same section we also briefly comment the case of more general equations of the type
\[
(-\Delta_p)^s u=f+\lambda\,|u|^{q-2}\,u,
\]
see Subsection \ref{sec:general}.
We then conclude the paper with a couple of appendices: while the material of Appendix \ref{sec:B} is standard, Appendix \ref{sec:stein} contains the proof of an embedding property of Besov-type spaces (Proposition \ref{prop:yes}), which is crucially exploited in the proof of Theorem \ref{teo:high}.

\begin{ack}
We thank Assia Benabdallah for a useful discussion on heat kernels. 
An informal discussion with Sunra Mosconi and Marco Squassina has lead to an improvement of Theorem \ref{teo:high}, we wish to thank them. We are grateful to Matteo Cozzi for pointing out to us his paper \cite{Co}. We also like to thank an anonymous referee for reading the paper carefully and for coming with many helpful comments and suggestions.
Part of this work has been done during a visit of the first author to Stockholm and of the second author to Marseille. The following institutions and their facilities are kindly acknowledged: Department of Mathematics of KTH (Stockholm), FRUMAM (Marseille) and CPT (Marseille).
\par 
L. B. is a member of the {\it Gruppo Nazionale per l'Analisi Matematica, la Probabilit\`a
e le loro Applicazioni} (GNAMPA) of the Istituto Nazionale di Alta Matematica (INdAM).
E. L. has been supported by the Swedish Research Council, grant no. 2012-3124.
\end{ack}

\section{Preliminaries}
\label{sec:preliminaries}

\subsection{Notation}
 Let $1\leq p<\infty$ and $0<\alpha<1$. For an open set $\Omega\subset\mathbb{R}^N$,
we denote by $W^{\alpha,p}(\Omega)$ the usual fractional Sobolev space defined as the set of functions such that
\[
\|\psi\|_{W^{\alpha,p}(\Omega)}:=[\psi]_{W^{\alpha,p}(\Omega)}+\|\psi\|_{L^p(\Omega)}<+\infty.
\]
The quantity $[\,\cdot\,]_{W^{\alpha,p}(\Omega)}$ is the $W^{\alpha,p}$ Gagliardo seminorm, i.e.
$$
[\psi]_{W^{\alpha,p}(\Omega)}=\left(\int_{\Omega}\int_{\Omega} \frac{|\psi(x)-\psi(y)|^{p}}{|x-y|^{N+\alpha\,p}}\, dx\, dy\right)^\frac{1}{p}.
$$
The local variant $W^{\alpha,p}_{loc}(\Omega)$ is defined in a straightforward manner. 
Given $h\in\mathbb{R}^N\setminus\{0\}$, for a measurable function $\psi:\mathbb{R}^N\to\mathbb{R}$ we introduce the notation
\[
\psi_h(x):=\psi(x+h)\qquad \mbox{ and }\qquad \delta_h \psi(x):=\psi_h(x)-\psi(x).
\]
We recall that for every pair of functions $\varphi,\psi$ we have
\begin{equation}
\label{leibniz}
\delta_h (\varphi\,\psi)=(\delta_h \varphi)\,\psi+\varphi_h\,(\delta_h \psi).
\end{equation}
We also remind the notation $\delta^2_h$ for the second order differences of a function, i.e.
\begin{equation}
\label{twotimes}
\delta^2_h\psi(x):=\delta_h(\delta_h \psi(x))=\psi(x+2\,h)-2\,\psi(x+h)+\psi(x).
\end{equation}
Finally, if $\Psi:\mathbb{R}^N\times\mathbb{R}^N\to\mathbb{R}$ is an integrable function on $A\times B$, the notation
\[
\int_A \int_B \Psi(x,y)\,dx\,dy,
\]
means that $x\in A$ and $y\in B$.
\subsection{Besov-type spaces}
The following spaces defined in terms of second order differences will be important\footnote{We recall that it is possible to consider the more general Besov space $\mathcal{B}^{\alpha,p}_q(\mathbb{R}^N)$, built up of $L^p$ functions such that
\[
\left(\int_{\mathbb{R}^N} \left(\int_{\mathbb{R}^N} \left|\frac{\delta^2_h \psi}{|h|^\alpha}\right|^p\,dx\right)^\frac{q}{p}\,\frac{dh}{|h|^N}\right)^\frac{1}{q}<+\infty.
\]
For $q=p$, we obtain the usual fractional Sobolev space $W^{\alpha,p}(\mathbb{R}^N)$. Also observe that our notation for Besov spaces is not the standard one: we prefer to adopt this in order to be consistent with that of $W^{\alpha,p}$.}.
\begin{defi}[Besov-Nikol'skii spaces]
Let $1\le p<\infty$ and $0<\alpha<2$. We say that $\psi\in\mathcal{B}^{\alpha,p}_{\infty}(\mathbb{R}^N)$ if
\[
\int_{\mathbb{R}^N} |\psi|^p\,dx<+\infty\qquad \mbox{ and }\qquad [\psi]_{\mathcal{B}^{\alpha,p}_\infty(\mathbb{R}^N)}^p:=\sup_{|h|>0}\int_{\mathbb{R}^N} \left|\frac{\delta^2_h \psi}{|h|^\alpha}\right|^p\,dx<+\infty.
\]
In this case, we set 
\[
\|\psi\|_{\mathcal{B}^{\alpha,p}_\infty(\mathbb{R}^N)}:= \|\psi\|_{L^p(\mathbb{R}^N)}+[\psi]_{\mathcal{B}^{\alpha,p}_\infty(\mathbb{R}^N)}.
\]
\end{defi}
We now need a couple of simple preliminary result for $\mathcal{B}^{\alpha,p}_\infty$. The first one states that it is indeed sufficient to control second order difference quotients for small translations. This is not surprising, we omit the proof.
\begin{lm}[Reduction to small translations]
Let $1\le p<\infty$ and $0<\alpha<2$. If $\psi\in \mathcal{B}^{\alpha,p}_{\infty}(\mathbb{R}^N)$ then for every $h_0>0$
\begin{equation}
\label{reductionh0}
[\psi]_{\mathcal{B}^{\alpha,p}_\infty(\mathbb{R}^N)}\le \left[\sup_{0<|h|<h_0}\left\|\frac{\delta^2_{h} \psi}{|h|^{\alpha}}\right\|_{L^p(\mathbb{R}^N)}+3\, h_0^{-\alpha}\, \left\|\psi\right\|_{L^p(\mathbb{R}^N)}\right].
\end{equation}
\end{lm}
In the case $0<\alpha<1$, second order difference quotients control first order ones\footnote{Actually, it is easy to see that they are equivalent in this range. Since we do not need the other estimate, we omit it.}. This is the content of the next result.
\begin{lm}
\label{lm:not}
Let $1\le p<\infty$ and $0<\alpha<1$. 
If $\psi\in \mathcal{B}^{\alpha,p}_{\infty}(\mathbb{R}^N)$ then
\begin{equation}
\label{reduction}
\sup_{|h|>0}\left\|\frac{\delta_{h} \psi}{|h|^{\alpha}}\right\|_{L^p(\mathbb{R}^N)}\le \frac{C}{1-\alpha}\,\left[[\psi]_{\mathcal{B}^{\alpha,p}_\infty(\mathbb{R}^N)}+\left\|\psi\right\|_{L^p(\mathbb{R}^N)}\right],
\end{equation}
for some universal constant $C>0$.
For every $h_0>0$, we also get
\begin{equation}
\label{reductionbis}
\sup_{0<|h|<h_0}\left\|\frac{\delta_{h} \psi}{|h|^{\alpha}}\right\|_{L^p(\mathbb{R}^N)}\le \frac{C}{1-\alpha}\,\left[\sup_{0<|h|<h_0}\left\|\frac{\delta^2_{h} \psi}{|h|^{\alpha}}\right\|_{L^p(\mathbb{R}^N)}+\left(h_0^{-\alpha}+1\right)\,\left\|\psi\right\|_{L^p(\mathbb{R}^N)}\right].
\end{equation}
\end{lm}
\begin{proof}
We will deduce the required estimate \eqref{reduction} by using some elementary manipulations, see also \cite[Chapter 2.6]{Tri}. We start by observing that for every measurable function $\psi$ we have
\[
\delta_h \psi(x)=\frac{1}{2}\, \Big(\delta_{2\,h}\psi (x)-\delta^2_{h}\psi(x)\Big).
\]
Thus for every $h\in\mathbb{R}^N\setminus \{0\}$ we get
\begin{equation}
\label{inequalitee}
\left\|\frac{\delta_h \psi}{|h|^\alpha}\right\|_{L^p(\mathbb{R}^N)}\le \frac{1}{2}\,\left\|\frac{\delta_{2\,h} \psi}{|h|^\alpha}\right\|_{L^p(\mathbb{R}^N)}+\frac{1}{2}\,\left\|\frac{\delta^2_h \psi}{|h|^\alpha}\right\|_{L^p(\mathbb{R}^N)},
\end{equation}
and observe that the second term on the right-hand side is uniformly bounded by the hypothesis. For the first one, we observe that if we set $h'=2\,h$
\[
\begin{split}
\left\|\frac{\delta_{2\,h} \psi}{|h|^\alpha}\right\|_{L^p(\mathbb{R}^N)}=2^\alpha\,\left\|\frac{\delta_{h'} \psi}{|h'|^\alpha}\right\|_{L^p(\mathbb{R}^N)}&\le 2^\alpha\,\sup_{0<|h'|<\frac{1}{2}} \left\|\frac{\delta_{h'} \psi}{|h'|^\alpha}\right\|_{L^p(\mathbb{R}^N)}+2^\alpha\,\sup_{\frac{1}{2}\le |h'|} \left\|\frac{\delta_{h'} \psi}{|h'|^\alpha}\right\|_{L^p(\mathbb{R}^N)}\\
&\le 2^\alpha\,\sup_{0<|h'|<\frac{1}{2}} \left\|\frac{\delta_{h'} \psi}{|h'|^\alpha}\right\|_{L^p(\mathbb{R}^N)}+2\cdot4^\alpha\, \left\|\psi\right\|_{L^p(\mathbb{R}^N)}.
\end{split}
\]
By using this estimate in \eqref{inequalitee}, we get
\[
\begin{split}
\sup_{0<|h|<\frac{1}{2}}\left\|\frac{\delta_h \psi}{|h|^\alpha}\right\|_{L^p(\mathbb{R}^N)}&\le \frac{1}{2}\,\sup_{0<|h|<\frac{1}{2}}\left\|\frac{\delta^2_h \psi}{|h|^\alpha}\right\|_{L^p(\mathbb{R}^N)}+4^\alpha\,\left\|\psi\right\|_{L^p(\mathbb{R}^N)}+2^{\alpha-1}\,\sup_{0<|h'|<\frac{1}{2}} \left\|\frac{\delta_{h'} \psi}{|h'|^\alpha}\right\|_{L^p(\mathbb{R}^N)}.
\end{split}
\]
By recalling that $\alpha<1$, the last term can be absorbed in the left-hand side and thus we get \eqref{reduction}. 
\par
Finally, estimate \eqref{reductionbis} can be obtained by combining \eqref{reduction} and \eqref{reductionh0}.
\end{proof}
The following result on Besov spaces will play a crucial role.  For the reader's convenience, we give a proof of this result in Appendix \ref{sec:stein}. The proof is essentially taken from Stein's book \cite{St} and is based on the so-called {\it thermic extension characterization} of Besov spaces (see \cite[Chapter 2.6]{Tri} for such a characterization). 
\begin{prop}
\label{prop:yes}
Let $1\le p<\infty$ and $1<\alpha<2$. We have the continuous embedding $\mathcal{B}^{\alpha,p}_{\infty}(\mathbb{R}^N)\hookrightarrow W^{1,p}(\mathbb{R}^N)$. In particular, for every $\psi\in \mathcal{B}^{\alpha,p}_\infty(\mathbb{R}^N)$ we have $\nabla \psi\in L^p(\mathbb{R}^N)$, with the following estimate
\begin{equation}
\label{02081980}
\|\nabla \psi\|_{L^p(\mathbb{R}^N)}\le C\,\|\psi\|_{L^p(\mathbb{R}^N)}+\frac{C}{(\alpha-1)}\, [\psi]_{\mathcal{B}^{\alpha,p}_\infty(\mathbb{R}^N)},
\end{equation}
for some constant $C=C(N,p)>0$. Moreover, we also have
\begin{equation}
\label{02082015}
\sup_{|h|>0}\left\|\frac{\delta_h \nabla \psi}{|h|^{\alpha-1}}\right\|_{L^p(\mathbb{R}^N)}\le \frac{C}{(2-\alpha)\,(\alpha-1)}\, [\psi]_{\mathcal{B}^{\alpha,p}_\infty(\mathbb{R}^N)},
\end{equation}
still for some $C=C(N,p)>0$.
\end{prop}
\begin{oss}
The previous result {\it is false} for the borderline case $\alpha=1$, see \cite[Example page 148]{St} for a counterexample.
\end{oss}
\subsection{Gagliardo seminorms and finite differences}
We still need a couple of basic facts on fractional order Sobolev spaces. The following results are well-known, but here we want to stress the explicit dependence of the constants on the differentiability index.
\begin{prop}
\label{lm:nikolski}
Let $1\le p<\infty$ and $0<\alpha<1$. 
\begin{itemize}
\item (Global case) For every $\psi\in W^{\alpha,p}(\mathbb{R}^N)$ there holds
\begin{equation}
\label{nikolski}
\sup_{|h|>0} \left\|\frac{\delta_h \psi}{|h|^{\alpha}}\right\|_{L^p(\mathbb{R}^N)}^p\le C\,(1-\alpha)\,[\psi]^p_{W^{\alpha,p}(\mathbb{R}^N)},
\end{equation}
for a constant $C=C(N,p)>0$;
\vskip.2cm
\item (Compactly supported case) Let $0<r<R$ and let $\psi\in W^{\alpha,p}(B_R)$ be such that $\psi \equiv 0$ on $B_R\setminus B_r$. Then we have
\begin{equation}
\label{nikolskib}
\sup_{|h|>0} \left\|\frac{\delta_h \psi}{|h|^{\alpha}}\right\|_{L^p(\mathbb{R}^N)}^p\le \frac{C}{\alpha}\,\left(\frac{R}{r}\right)^N\,\left(\frac{R}{R-r}\right)^{1+p}\,(1-\alpha)\,[\psi]^p_{W^{\alpha,p}(B_R)},
\end{equation}
where $\psi$ is extended by $0$ to the whole $\mathbb{R}^N$ and $C=C(N,p)>0$;
\vskip.2cm
\item (Local case) Let $\Omega\subset \mathbb{R}^N$ be an open set. Let $\psi\in W^{\alpha,p}_{loc}(\Omega)$, then for every ball $B_R\Subset \Omega$ and every $0<h_0\le \mathrm{dist}(B_R,\partial\Omega)/2$ we have
\begin{equation}
\label{nikolskibis}
\begin{split}
\sup_{0<|h|<h_0} \left\|\frac{\delta_h \psi}{|h|^{\alpha}}\right\|_{L^p(B_R)}^p&\le C\,(1-\alpha)\,[\psi]^p_{W^{\alpha,p}(B_{R+h_0})}\\
&+C\,\left[\left(1+\frac{R}{h_0}\right)^p\,(R+h_0)^{-\alpha\,p}+\frac{h_0^{-\alpha\,p}}{\alpha}\,\right]\,\|\psi\|^p_{L^p(B_{R+h_0})},
\end{split}
\end{equation}
for a constant $C=C(N,p)>0$.
\end{itemize}
\end{prop}
\begin{proof}
An elementary proof of \eqref{nikolski} can be found for example in \cite[Lemma A.1]{BLP}. In order to prove \eqref{nikolskib}, we decompose the Gagliardo seminorm of $\psi$ as follows
\begin{equation}
\label{onco}
[\psi]^p_{W^{\alpha,p}(\mathbb{R}^N)}=[\psi]_{W^{\alpha,p}(B_R)}^p+2\,\int_{B_R}\int_{\mathbb{R}^N\setminus B_R} \frac{|\psi(x)|^p}{|x-y|^{N+\alpha\,p}}\,dx\,dy.
\end{equation}
Then we observe that since $\psi\equiv 0$ in $B_R\setminus B_r$
\[
\begin{split}
\int_{B_R}\int_{\mathbb{R}^N\setminus B_R} \frac{|\psi(x)|^p}{|x-y|^{N+\alpha\,p}}\,dx\,dy&=\int_{B_r}\int_{\mathbb{R}^N\setminus B_R} \frac{|\psi(x)|^p}{|x-y|^{N+\alpha\,p}}\,dx\,dy\\
&\le \frac{C}{\alpha\,(R-r)^{\alpha\,p}}\,\int_{B_r} |\psi|^p\,dx.
\end{split}
\]
The last term can be estimated by appealing to the Poincar\'e inequality (see \cite[Proposition 2.2]{BP})
\[
C\,\left(\frac{r}{R}\right)^N\,\left(\frac{R-r}{R}\right)\, \frac{1}{R^{\alpha\,p}}\,\int_{B_r} |\psi|^p\,dx\le [\psi]^p_{W^{\alpha,p}(B_R)},
\]
where $C=C(N,p)>0$. By inserting the resulting estimate in \eqref{onco} and combining with \eqref{nikolski}, we get \eqref{nikolskib} as desired. Observe that we also used that $R/(R-r)>1$, in order to replace a power $1+\alpha\,p$ by $1+p$.
Finally, for \eqref{nikolskibis} we first take a standard Lipschitz cut-off function $\eta$ such that
\[
0\le \eta\le 1,\qquad \eta\equiv 1 \mbox{ on } B_R,\qquad \eta\equiv 0\mbox{ on } \mathbb{R}^N\setminus B_{R+\frac{h_0}{2}},\qquad |\nabla \eta|\le \frac{2}{h_0}.
\]
Then we observe that $\psi\,\eta\in W^{\alpha,p}(\mathbb{R}^N)$, thus by using the discrete Leibniz rule \eqref{leibniz}, \eqref{nikolski} and the properties of $\eta$ we get
\[
\begin{split}
\sup_{0<|h|<h_0} \left\|\frac{\delta_h \psi}{|h|^{\alpha}}\right\|_{L^p(B_R)}^p&\le 2^{p-1}\,\sup_{0<|h|<h_0} \left\|\frac{\delta_h (\psi\,\eta)}{|h|^{\alpha}}\right\|_{L^p(B_R)}^p+2^{p-1}\,\sup_{0<|h|<h_0} \left\|\frac{\delta_h \eta}{|h|^{\alpha}}\,\psi_h\right\|_{L^p(B_R)}^p\\
&\le C\,(1-\alpha)\,[\psi\,\eta]^p_{W^{\alpha,p}(\mathbb{R}^N)}+C\, h_0^{-\alpha\,p}\,\|\psi\|^p_{L^p(B_{R+h_0})}.
\end{split}
\]
We proceed as before
\[
\begin{split}
[\psi\,\eta]^p_{W^{\alpha,p}(\mathbb{R}^N)}&=[\psi\,\eta]^p_{W^{\alpha,p}(B_R+h_0)}+2\,\int_{B_{R+h_0}} \int_{\mathbb{R}^N\setminus B_{R+h_0}} \frac{|\psi(x)|^{p}\,\eta(x)^p}{|x-y|^{N+\alpha\,p}}\,dx\,dy\\
&\le C\,[\psi]^p_{W^{\alpha,p}(B_{R+h_0})}+C\,\int_{B_{R+h_0}}\int_{B_{R+h_0}} \frac{|\eta(x)-\eta(y)|^p}{|x-y|^{N+\alpha\,p}}\,|\psi(y)|^p\,dx\,dy\\
&+2\,\int_{B_{R+\frac{h_0}{2}}} \int_{\mathbb{R}^N\setminus B_{R+h_0}} \frac{|\psi(x)|^{p}}{|x-y|^{N+\alpha\,p}}\,dx\,dy.
\end{split}
\]
By using the Lipschitz character of $\eta$, we can now easily get \eqref{nikolskibis}.
\end{proof}
\begin{prop}
\label{prop:lostinpassing}
Let $1\le p<\infty$ and $0<\alpha<\beta\le 1$. Let $\psi\in L^p(\mathbb{R}^N)$ be such that for some $h_0>0$ we have
\[
\sup_{0<|h|<h_0} \left\|\frac{\delta_h \psi}{|h|^{\beta}}\right\|_{L^p(\mathbb{R}^N)}^p<+\infty.
\]
Then there holds
\[
[\psi]^p_{W^{\alpha,p}(\mathbb{R}^N)}\le C\,\left(\frac{h_0^{(\beta-\alpha)\,p}}{\beta-\alpha}\, \sup_{0<|h|<h_0} \left\|\frac{\delta_h \psi}{|h|^{\beta}}\right\|_{L^p(\mathbb{R}^N)}^p+\frac{h_0^{-\alpha\,p}}{\alpha}\,\|\psi\|^p_{L^p(\mathbb{R}^N)}\right),
\]
for some constant $C=C(N,p)>0$.
\end{prop}
\begin{proof}
The proof is elementary, we give it for completeness. We have
\[
\begin{split}
[\psi]^p_{W^{\alpha,p}(\mathbb{R}^N)}&=\int_{\{|h|<h_0\}} \int_{\mathbb{R}^N} \frac{|\delta_h \psi(x)|^p}{|h|^{N+\alpha\,p}}\,dh\,dx+\int_{\{|h|\ge h_0\}} \int_{\mathbb{R}^N} \frac{|\delta_h \psi(x)|^p}{|h|^{N+\alpha\,p}}\,dh\,dx\\
&\le \int_{\{|h|<h_0\}} \left(\int_{\mathbb{R}^N} \frac{|\delta_h \psi(x)|^p}{|h|^{\beta\,p}}\,dx\right)\,\frac{dh}{|h|^{N-(\beta-\alpha)\,p}}+2^{p-1}\,\|\psi\|^{p}_{L^p(\mathbb{R}^N)}\,\int_{\{|h|\ge h_0\}} \frac{1}{|h|^{N+\alpha\,p}}\,dh\\
&\le\frac{C\,h_0^{(\beta-\alpha)\,p}}{(\beta-\alpha)}\, \sup_{0<|h|<h_0} \left\|\frac{\delta_h \psi}{|h|^\beta}\right\|_{L^p(\mathbb{R}^N)}^p+\frac{C\,h_0^{-\alpha\,p}}{\alpha}\,\|\psi\|^p_{L^p(\mathbb{R}^N)}.
\end{split}
\]
The constant $C$ above depends on $N$ and $p$ only. This concludes the proof.
\end{proof}
\subsection{Special spaces}
In this subsection, we present some basic properties of the spaces $\mathcal{X}^p_s$ and $\mathcal{Y}^{t,p}_s$ we introduced in Definition \ref{defi:snails}. 
We recall the notation
\[
d(F,E):=\mathrm{dist}(F,\mathbb{R}^N\setminus E).
\]
\begin{lm}[Inclusions]
Let $1\le p<\infty$ and $0<s<1$ and $\Omega\subset\mathbb{R}^N$. Then we have the following inclusions
\begin{equation}
\label{Lp}
L^q(\mathbb{R}^N)\subset \mathcal{X}^p_s,\qquad \mbox{ for every }q\ge p,
\end{equation}
\begin{equation}
\label{Wsp}
W^{s,p}(\mathbb{R}^N)\subset \mathcal{Y}^{t,p}_{s},\qquad \mbox{ for every }0\le t\le s,
\end{equation}
\begin{equation}
\label{Wsploc}
\mathcal{Y}^{s,p}_{s}\subset W^{\tau,p}_{loc}(\mathbb{R}^N),\qquad \mbox{ for every } 0<\tau<s.
\end{equation}
\end{lm}
\begin{proof}
The first inclusion \eqref{Lp} stems from the simple observation that by H\"older inequality we have
\[
\int_{\mathbb{R}^N} \frac{|\psi(x)|^p}{(1+|x|)^{N+s\,p}}\,dx\le \|\psi\|^p_{L^q(\mathbb{R}^N)}\,\left(\int_{\mathbb{R}^N} (1+|x|)^{-\frac{q}{q-p}\,(N+s\,p)}\,dx\right)^\frac{q-p}{q}<+\infty.
\]
Similarly, for the second inclusion \eqref{Wsp} we observe that for every $h_0>0$ we have
\[
\sup_{0<|h|<h_0}\int_{\mathbb{R}^N}\left|\frac{\delta_h \psi}{|h|^s}\right|^p\,\frac{dx}{(1+|x|)^{N+s\,p}}\le \sup_{0<|h|<h_0}\int_{\mathbb{R}^N} \left|\frac{\delta_h \psi}{|h|^s}\right|^p\,dx,
\]
and the last term is bounded by the $W^{s,p}$ seminorm of $\psi$, thanks to \eqref{nikolski}. This shows \eqref{Wsp} for $t=s$, the general case follows by observing that $\mathcal{Y}^{s,p}_s\subset \mathcal{Y}^{t,p}_s$ for $0\le t<s$.
\vskip.2cm\noindent
Finally, we prove \eqref{Wsploc}. Let $\psi\in \mathcal{Y}^{s,p}_s$,
by definition of $\mathcal{Y}^{s,p}_p$ there exists $h_0>0$ such that 
\[
\sup_{0<|h|<h_0}\int_{\mathbb{R}^N} \left|\frac{\delta_h \psi(y)}{|h|^s}\right|^p\,\frac{1}{(1+|x|)^{N+s\,p}}\,dx<+\infty,
\]
thus in particular for every open and bounded set $\mathcal{O}\subset\mathbb{R}^N$ we have
\[
\left(1+\sup\limits_{x\in\mathcal{O}}|x|\right)^{-N-s\,p}\sup_{0<|h|<h_0}\left\|\frac{\delta_h \psi}{|h|^s}\right\|^p_{L^p(\mathcal{O})}<+\infty.
\]
We now get the conclusion by proceeding as in the proof of Proposition \ref{prop:lostinpassing}. We leave the details to the reader.
\end{proof}
The following monotonicity properties will be needed in the proof of Theorem \ref{teo:high}.
\begin{lm}[Snails monotonicity]
\label{lm:monotone}
Let $1\le p<\infty$ and $0<s<1$. We consider two pairs of sets $F_1\Subset E_1\Subset\mathbb{R}^N$ and $F_2\Subset E_2\Subset\mathbb{R}^N$ such that
\[
F_1\subset F_2\qquad \mbox{ and }\qquad E_1\subset E_2.
\]
Then for every $\psi\in \mathcal{X}^p_s$ we have
\begin{equation}
\label{snailing0}
\int_{F_1} \mathrm{Snail}(\psi;x,E_1)^p\,dx\le \left(\frac{|E_1|}{|E_2|}\right)^\frac{s\,p}{N}\,\left[\int_{F_2} \mathrm{Snail}(\psi;x,E_2)^p\,dx+\frac{|F_1|\,|E_2|^\frac{s\,p}{N}}{d(F_1,E_1)^{N+s\,p}}\,\int_{E_2\setminus E_1} |\psi|^p\,dy\right].
\end{equation}
In particular, we get
\begin{equation}
\label{snailing}
\langle \psi\rangle^p_{\mathcal{X}^p_s(F_1;E_1)}\le \left(\frac{|E_1|}{|E_2|}\right)^\frac{s\,p}{N}\,\left(1+\frac{|F_1|\,|E_2|^\frac{s\,p}{N}}{d(F_1,E_1)^{N+s\,p}}\right)\,\langle \psi\rangle_{\mathcal{X}^p_s(F_2;E_2)}^p,
\end{equation}
\end{lm}
\begin{proof}
The proof of \eqref{snailing0} is elementary. We have
\[
\begin{split}
\int_{F_1} \mathrm{Snail}(\psi;x,E_1)^p\,dx&=|E_1|^\frac{s\,p}{N}\,\int_{F_1}\int_{\mathbb{R}^N\setminus E_1} \frac{|\psi(y)|^p}{|x-y|^{N+s\,p}}\,dx\,dy\\
&= \left(\frac{|E_1|}{|E_2|}\right)^\frac{s\,p}{N}\,|E_2|^\frac{s\,p}{N}\,\int_{F_1}\int_{\mathbb{R}^N\setminus E_2} \frac{|\psi(y)|^p}{|x-y|^{N+s\,p}}\,dx\,dy\\
&+\left(\frac{|E_1|}{|E_2|}\right)^\frac{s\,p}{N}\,|E_2|^\frac{s\,p}{N}\, \int_{F_1}\int_{E_2\setminus E_1}\frac{|\psi(y)|^p}{|x-y|^{N+s\,p}}\,dx\,dy\\
&\le \left(\frac{|E_1|}{|E_2|}\right)^\frac{s\,p}{N}\,\int_{F_2} \mathrm{Snail}(\psi;x,E_2)^p\,dx\\
&+\left(\frac{|E_1|}{|E_2|}\right)^\frac{s\,p}{N}\,\frac{|F_1|\,|E_2|^\frac{s\,p}{N}}{d(F_1,E_1)^{N+s\,p}}\,\int_{E_2\setminus E_1} |\psi|^p\,dy.
\end{split}
\]
With some standard manipulations we get \eqref{snailing} as well.
\end{proof}

\section{Basic estimates}
\label{sec:evil}

Throughout the whole section, we denote by $u\in W^{s,p}_{loc}(\Omega)\cap \mathcal{Y}^{t,p}_s$ a local weak  solution of \eqref{localK}, with right-hand side $f\in W^{s,p'}_{loc}(\Omega)$ and $K$ satisfying \eqref{Kassummption}. Thus for every $\Omega'\Subset \Omega$ and any $\varphi\in W^{s,p}(\Omega')$ such that $\varphi\equiv 0$ on $\Omega\setminus\Omega'$, the function $u$ satisfies \eqref{weak}. For notational simplicity, we will set
\begin{equation}
\label{mu}
d\mu=\frac{1}{K(x-y)}\,dx\,dy.
\end{equation}
We also set
\[
J_p(t)=|t|^{p-2}\,t\qquad \mbox{ and }\qquad V_p(t)=|t|^\frac{p-2}{2}\,t,
\]
and then define the nonlinear function of the solution $\mathcal{V}_p:\mathbb{R}^N\times\mathbb{R}^N\to\mathbb{R}$ by
\begin{equation}
\label{Vp}
\mathcal{V}_p(x,y)=V_p\left(u(x)-u(y)\right)=|u(x)-u(y)|^\frac{p-2}{2}\,(u(x)-u(y)).
\end{equation}
By a slight abuse of notation, for every $h\in\mathbb{R}^N\setminus{0}$ we will use the following convention
\[
\begin{split}
\delta_h \mathcal{V}_p(x,y)&=\left(\mathcal{V}_p\right)_h(x,y)-\mathcal{V}_p(x,y)=V_p\left(u_h(x)-u_h(y)\right)-V_p\left(u(x)-u(y)\right).
\end{split}
\]
\subsection{Caccioppoli-type inequality}
We start with the following general estimate containing a free parameter of differentiability $\gamma$. 
This is an iterative scheme which improves the differentiability of $u$.
We notice that the case $s=t=\gamma=1$ formally corresponds to the result \eqref{classic} for the $p-$Laplacian.
\begin{prop}[Differentiability scheme]
\label{prop:derivatives}
Let $p\ge 2$, $0<s<1$ and $0\le t\le s$. We take $B_r\Subset B_R\Subset\Omega$ a pair of concentric balls and fix 
\[
0<h_0<\frac{1}{4}\,\min\Big\{\mathrm{dist}(B_{R};\partial\Omega),\, R-r,\,1\Big\}.
\] 
We take $\eta$ a standard $C^2$ cut-off function such that
\[
0\le \eta\le1,\qquad \eta\equiv 1 \quad \mbox{ on } B_r,\qquad \eta\equiv 0 \quad \mbox{ on } \mathbb{R}^N\setminus B_\frac{R+r}{2},\qquad |\nabla \eta|\le \frac{c_N}{R-r},\qquad |D^2 \eta|\le \frac{c_N}{(R-r)^2}.
\] 
For every $h\in\mathbb{R}^N\setminus\{0\}$ such that $|h|<h_0$ and every $s\le \gamma \le 1$ we have 
\begin{equation}
\label{regularity}
\begin{split}
\left[\frac{\delta_h (u\,\eta)}{|h|^\frac{\gamma+t}{p}}\right]^p_{W^{s,p}(B_R)}&\le \left(\frac{R}{R-r}\right)^p\,
\frac{C}{(1-s)\,s}\,\frac{1}{(R-r)^{s\,p}}\,\left\|\frac{\delta_h u}{|h|^\gamma}\right\|^p_{L^p(B_R)}\\
&+C\,\left(\frac{R}{R-r}\right)^p\,h_0^{-\gamma-t}\,\left(\left[u\right]^p_{W^{s,p}(B_{R+h_0})}+\frac{1}{s\,(1-s)\,R^{s\,p}}\,\|u\|^p_{L^p(B_{R+h_0})}\right)\\
&+\frac{C}{s}\,\left(\frac{R+r}{R-r}\right)^N\,\left(\frac{R+1}{R-r}\right)^{s\,p}\,\frac{1}{(R-r)^{s\,p}}\,\langle u\rangle^p_{\mathcal{X}^{p}_s(B_{\frac{R+r}{2}+h_0},B_{R+h_0})}\\
&+\frac{C}{R^{s\,p}}\,\langle u\rangle^p_{\mathcal{Y}^{t,p}_s(B_\frac{R+r}{2};B_R)}+C\,(1-s)^\frac{1}{p-1}\,R^{s\,p'}\,\left\|\frac{\delta_h f}{|h|^s}\right\|^{p'}_{L^{p'}(B_R)}\\
\end{split}
\end{equation}
for a constant $C=C(N,p,\Lambda)>0$.
\end{prop}
\begin{proof}
We take a test function $\varphi\in W^{s,p}(\Omega)$ such that $\varphi\equiv 0$ on $\Omega\setminus B_{(R+r)/2}$. 
By testing \eqref{weak} with $\varphi_{-h}$ for $h\in\mathbb{R}^N\setminus\{0\}$ with $|h|<h_0$ and then changing variables, we get
\begin{equation}
\label{equationh}
\int_{\mathbb{R}^N}\int_{\mathbb{R}^N} \Big(J_p(u_h(x)-u_h(y))\Big)\,\Big(\varphi(x)-\varphi(y)\Big)\,d\mu(x,y)=\int_\Omega f_h\,\varphi\,dx.
\end{equation}
We recall that $\mu$ is the singular measure defined in \eqref{mu}.
We now subtract \eqref{weak} from \eqref{equationh}, thus we get
\begin{equation}
\label{differentiated}
\begin{split}
\int_{\mathbb{R}^N}\int_{\mathbb{R}^N} \Big(J_p(u_h(x)-u_h(y))-J_p(u(x)-u(y))\Big)\,\Big(\varphi(x)-\varphi(y)\Big)\,d\mu(x,y)=\int_\Omega \delta_h f\,\varphi\,dx,
\end{split}
\end{equation}
for every $\varphi \in W^{s,p}(B_{(R+r)/2})$ such that $\varphi\equiv 0$ on $\Omega\setminus B_{(R+r)/2}$.
Finally, we insert in \eqref{differentiated} the test function
\[
\varphi=\frac{\delta_h u}{|h|^{\gamma+t}}\, \eta^p,
\]
where $\eta$ is the cut-off function of the statement.
We now divide the double integral in \eqref{differentiated} in three pieces:
\[
\begin{split}
\mathcal{I}_1:=\int_{B_R}\int_{B_R}& \frac{\Big(J_p(u_h(x)-u_h(y))-J_p(u(x)-u(y))\Big)}{|h|^{\gamma+t}}\,\Big(\delta_h u(x)\,\eta(x)^p-\delta_h u(y)\,\eta(y)^p\Big)\,d\mu(x,y),
\end{split}
\]
\[
\begin{split}
\mathcal{I}_2:=\int_{B_R}\int_{\mathbb{R}^N\setminus B_R}& \frac{\Big(J_p(u_h(x)-u_h(y))-J_p(u(x)-u(y))\Big)}{|h|^{\gamma+t}}\, \delta_h u_h(x)\,\eta(x)^p\,d\mu(x,y),
\end{split}
\]
and
\[
\begin{split}
\mathcal{I}_3:=-\int_{\mathbb{R}^N\setminus B_R}\int_{B_R}& \frac{\Big(J_p(u_h(x)-u_h(y))-J_p(u(x)-u(y))\Big)}{|h|^{\gamma+t}}\,\delta_h u(y)\,\eta(y)^p\,d\mu(x,y),
\end{split}
\]
We estimate each term separately.
\vskip.2cm\noindent
{\bf Estimate of $\mathcal{I}_1$.}
We start by observing that
\[
\begin{split}
\delta_h u(x)\,\eta(x)^p-\delta_h u(y)\,\eta(y)^p&=\frac{\Big(\delta_h u(x)-\delta_h u(y)\Big)}{2}\,\Big(\eta(x)^p+\eta(y)^p\Big)\\
&+\frac{\Big(\delta_h u(x)+\delta_h u(y)\Big)}{2}\,(\eta(x)^p-\eta(y)^p).
\end{split}
\]
Thus we get
\[
\begin{split}
\Big(J_p(u_h(x)-u_h(y))&-J_p(u(x)-u(y))\Big)\,\Big((\delta_h u(x))\,\eta(x)^p-(\delta_h u(y))\,\eta(y)^p\Big)\\
&\ge \Big(J_p(u_h(x)-u_h(y))-J_p(u(x)-u(y))\Big)\,\Big(\delta_h u(x)-\delta_h u(y)\Big)\,\left(\frac{\eta(x)^p+\eta(y)^p}{2}\right)\\
&-\Big|J_p(u_h(x)-u_h(y))-J_p(u(x)-u(y))\Big|\,\Big(|\delta_h u(x)|+|\delta_h u(y)|\Big)\,\left|\frac{\eta(x)^p-\eta(y)^p}{2}\right|.\\
\end{split}
\]
The first term has a positive sign and we will keep it on the left-hand side. For the negative term, we proceed as follows: we use \eqref{lipschitz}, the definition \eqref{Vp} of $\mathcal{V}_p$, Young inequality and \eqref{monotone} to get
\[
\begin{split}
\Big|J_p(u_h(x)-u_h(y))&-J_p(u(x)-u(y))\Big|\Big(|\delta_h u(x)|+|\delta_h u(y)|\Big)\,\left|\frac{\eta(x)^p-\eta(y)^p}{2}\right|\\
&\le 2\,\frac{p-1}{p}\,\left(|u_h(x)-u_h(y)|^\frac{p-2}{2}+|u(x)-u(y)|^\frac{p-2}{2}\right)\\
&\times |\delta_h \mathcal{V}_p(x,y)|\, \Big(|\delta_h u(x)|+|\delta_h u(y)|\Big)\,\frac{\eta(x)^\frac{p}{2}+\eta(y)^\frac{p}{2}}{2}\,\left|\eta(x)^\frac{p}{2}-\eta(y)^\frac{p}{2}\right|\\
&\le \frac{C}{\varepsilon}\,\left(|u_h(x)-u_h(y)|^\frac{p-2}{2}+|u(x)-u(y)|^\frac{p-2}{2}\right)^2\\
&\times (|\delta_h u(x)|^2+|\delta_h u(y)|^2)\,\left|\eta(x)^\frac{p}{2}-\eta(y)^\frac{p}{2}\right|^2\\
&+C\,\varepsilon\,|\delta_h \mathcal{V}_p(x,y)|^2\,\Big(\eta(x)^p+\eta(y)^p\Big)\\
&\le \frac{C}{\varepsilon}\,\left(|u_h(x)-u_h(y)|^\frac{p-2}{2}+|u(x)-u(y)|^\frac{p-2}{2}\right)^2\,(|\delta_h u(x)|^{2}+|\delta_h u(y)|^{2})\,\left|\eta(x)^\frac{p}{2}-\eta(y)^\frac{p}{2}\right|^2\\
&+C\,\varepsilon\,\Big(J_p(u_h(x)-u_h(y))-J_p(u(x)-u(y))\Big)\,\Big(\delta_h u(x)-\delta_h u(y)\Big)\,\Big(\eta(x)^p+\eta(y)^p\Big),
\end{split}
\]
where $C=C(p)>0$.
By putting all the estimates together and choosing $\varepsilon$ sufficiently small, we then get
\[
\begin{split}
\mathcal{I}_1&\ge \frac{1}{C}\,\int_{B_R}\int_{B_R}  \frac{J_p(u_h(x)-u_h(y))-J_p(u(x)-u(y))}{|h|^{\gamma+t}}\,\Big(\delta_h u(x)-\delta_h u(y)\Big)\,(\eta(x)^p+\eta(y)^p)\,d\mu(x,y)\\
&-C\,\int_{B_R}\int_{B_R} \left(|u_h(x)-u_h(y)|^\frac{p-2}{2}+|u(x)-u(y)|^\frac{p-2}{2}\right)^2\,\left|\eta(x)^\frac{p}{2}-\eta(y)^\frac{p}{2}\right|^2\\
&\times \frac{|\delta_h u(x)|^{2}+|\delta_h u(y)|^{2}}{|h|^{\gamma+t}}\,d\mu(x,y),
\end{split}
\]
for some constant $C=C(p)>0$.
We can further estimate from below the positive term by using \eqref{down}. This leads us to
\begin{equation}
\label{I1V}
\begin{split}
\mathcal{I}_1&\ge \frac{1}{C}\,\int_{B_R}\int_{B_R} \left|\frac{\delta_h u(x)}{|h|^\frac{\gamma+t}{p}}-\frac{\delta_h u(y)}{|h|^\frac{\gamma+t}{p}}\right|^p\,(\eta(x)^p+\eta(y)^p)\,d\mu(x,y)\\
&-C\,\int_{B_R}\int_{B_R} \left(|u_h(x)-u_h(y)|^\frac{p-2}{2}+|u(x)-u(y)|^\frac{p-2}{2}\right)^2\,\left|\eta(x)^\frac{p}{2}-\eta(y)^\frac{p}{2}\right|^2\\
&\times\frac{|\delta_h u(x)|^2+|\delta_h u(y)|^2}{|h|^{\gamma+t}}\,d\mu(x,y).
\end{split}
\end{equation}
We now observe that if we set for simplicity
\[
A=\frac{\delta_h u(x)}{|h|^\frac{\gamma+t}{p}}\qquad \mbox{ and }\qquad B=\frac{\delta_h\,u(y)}{|h|^\frac{\gamma+t}{p}},
\]
then by using the convexity of $\tau\mapsto |\tau|^p$, we have
\[
\begin{split}
\left|A\,\eta(x)-B\,\eta(y)\right|^p=\left|(A-B)\,\frac{\eta(x)+\eta(y)}{2}+(A+B)\,\frac{\eta(x)-\eta(y)}{2}\right|^p&\le 2^{p-2}\,|A-B|^p\, (\eta(x)^p+\eta(y)^p)\\
&+2^{p-2}\, (|A|^p+|B|^p)\, |\eta(x)-\eta(y)|^p.
\end{split}
\]
Thus from \eqref{I1V} together with the assumption \eqref{Kassummption} on $K$, we get the following lower bound for $\mathcal{I}_1$
\begin{equation}
\label{I1VI}
\begin{split}
\mathcal{I}_1&\ge \frac{1}{C}\,\left[\frac{\delta_h u}{|h|^\frac{\gamma+t}{p}}\,\eta\right]^p_{W^{s,p}(B_R)}\\
&-C\,\int_{B_R}\int_{B_R} \left(|u_h(x)-u_h(y)|^\frac{p-2}{2}+|u(x)-u(y)|^\frac{p-2}{2}\right)^2\,\left|\eta(x)^\frac{p}{2}-\eta(y)^\frac{p}{2}\right|^2\\
&\times \frac{|\delta_h u(x)|^2+|\delta_h u(y)|^2}{|h|^{\gamma+t}}\,d\mu(x,y)\\
&-C\,\int_{B_R}\int_{B_R}\, \left(\frac{|\delta_h u(x)|^{p}}{|h|^{\gamma+t}}+\frac{|\delta_h u(y)|^{p}}{|h|^{\gamma+t}}\right)\, \frac{|\eta(x)-\eta(y)|^p}{|x-y|^{N+s\,p}}\,dx\,dy,
\end{split}
\end{equation}
where $C=C(p,\Lambda)>0$. We need to estimate the last two integrals. For the first one, by using H\"older's inequality, again the assumption \eqref{Kassummption} on $K$, the Lipschitz character of $\eta$ and some simple manipulations we get
\[
\begin{split}
\int_{B_R}\int_{B_R} &\left(|u_h(x)-u_h(y)|^\frac{p-2}{2}+|u(x)-u(y)|^\frac{p-2}{2}\right)^2\,\left|\eta(x)^\frac{p}{2}-\eta(y)^\frac{p}{2}\right|^2\\
&\times \frac{|\delta_h u(x)|^{2}+|\delta_h u(y)|^{2}}{|h|^{\gamma+t}}\,d\mu(x,y)\\
&\le C\,\left[\int_{B_R}\int_{B_R} \left(|u_h(x)-u_h(y)|^{p-2}+|u(x)-u(y)|^{p-2}\right)^\frac{p}{p-2}\,d\mu(x,y)\right]^\frac{p-2}{p}\\
&\times \left[\int_{B_R}\int_{B_R} \frac{|\delta_h u(x)|^{p}+|\delta_h u(y)|^{p}}{|h|^{(\gamma+t)\,\frac{p}{2}}}\,\left|\eta(x)^\frac{p}{2}-\eta(y)^\frac{p}{2}\right|^p\,d\mu(x,y)\right]^\frac{2}{p}\\
&\le \frac{C}{(R-r)^2}\, [u]^{p-2}_{W^{s,p}(B_{R+h_0})}\,\left[\int_{B_R} \left|\frac{\delta_h u(x)}{|h|^\frac{\gamma+t}{2}}\right|^p\,\left(\int_{B_R} \frac{1}{|x-y|^{N+p\,(s-1)}}\,dy\right)dx\right]^\frac{2}{p}\\
&\le \frac{C}{R^{2\,s}}\,\left(\frac{R}{R-r}\right)^2\, [u]^{p-2}_{W^{s,p}(B_{R+h_0})}\,\left[\frac{1}{1-s}\,\int_{B_R} \left|\frac{\delta_h u}{|h|^\frac{\gamma+t}{2}}\right|^p\,dx\right]^\frac{2}{p}\\
&\le C\,[u]^{p}_{W^{s,p}(B_{R+h_0})}+\left(\frac{R}{R-r}\right)^p \,\frac{C}{R^{s\,p}}\,\frac{1}{1-s}\,\int_{B_R} \left|\frac{\delta_h u}{|h|^\frac{\gamma+t}{2}}\right|^p\,dx,
\end{split}
\]
for some $C=C(N,p,\Lambda)>0$. Thus, from \eqref{I1VI} by observing that $|h|<h_0<1$ and that $(\gamma+t)/2\le \gamma$, we get
\begin{equation}
\label{I10}
\begin{split}
\mathcal{I}_1&\ge \frac{1}{C}\,\left[\frac{\delta_h u}{|h|^\frac{\gamma+t}{p}}\,\eta\right]^p_{W^{s,p}(B_R)}\\
&-C\,[u]^{p}_{W^{s,p}(B_{R+h_0})}-\left(\frac{R}{R-r}\right)^p\, \frac{C}{(R-r)^{s\,p}}\,\frac{1}{s\,(1-s)}\,\int_{B_R} \left|\frac{\delta_h u}{|h|^\gamma}\right|^p\,dx\\
&-C\,\int_{B_R}\int_{B_R}\, \left(\frac{|\delta_h u(x)|^{p}}{|h|^{\gamma+t}}+\frac{|\delta_h u(y)|^{p}}{|h|^{\gamma+t}}\right)\, \frac{|\eta(x)-\eta(y)|^p}{|x-y|^{N+s\,p}}\,dx\,dy,
\end{split}
\end{equation}
where we also used that $R^{s\,p}\ge (R-r)^{s\,p}$ and that $s\,(1-s)\le (1-s)$. By the Lipschitz character of $\eta$, the last integral is estimated by
\begin{equation}
\label{eta}
\begin{split}
\int_{B_R}\int_{B_R}\,& \left(\frac{|\delta_h u(x)|^{p}}{|h|^{\gamma+t}}+\frac{|\delta_h u(y)|^{p}}{|h|^{\gamma+t}}\right)\, \frac{|\eta(x)-\eta(y)|^p}{|x-y|^{N+s\,p}}\,dx\,dy\\
&\le \frac{C}{(R-r)^{s\,p}}\,\left(\frac{R}{R-r}\right)^p\,\frac{1}{s\,(1-s)}\,\int_{B_R} \left|\frac{\delta_h u}{|h|^{\gamma}}\right|^p\,dx,
\end{split}
\end{equation}
for some $C=C(N,p)>0$. Observe that we again used the trivial estimates $R^{s\,p}\ge (R-r)^{s\,p}$ and $s\,(s-1)\le (s-1)$, together with $h_0<1$ and $(\gamma+t)/p\le \gamma$.
\par
It is only left to observe that from the discrete Leibniz rule \eqref{leibniz}, we get
\[
\begin{split}
\left[\frac{\delta_h (u\,\eta)}{|h|^\frac{\gamma+t}{p}}\right]^p_{W^{s,p}(B_R)}&\le C\,
\left[\frac{\delta_h u}{|h|^\frac{\gamma+t}{p}}\,\eta\right]^p_{W^{s,p}(B_R)}+C\,
\left[\frac{\delta_h \eta}{|h|^\frac{\gamma+t}{p}}\,u_h\right]^p_{W^{s,p}(B_R)}\\
&\le C\,\left[\frac{\delta_h u}{|h|^\frac{\gamma+t}{p}}\,\eta\right]_{W^{s,p}(B_R)}\\
&+C\,
\left(\frac{R}{R-r}\right)^p\,\frac{h_0^{p-\gamma-t}}{(R-r)^p}\,\left(\left[u\right]^p_{W^{s,p}(B_{R+h_0})}+\frac{1}{(1-s)\,R^{s\,p}}\,\|u\|^p_{L^p(B_{R+h_0})}\right).
\end{split}
\]
Observe that by the hypothesis $h_0/(R-r)<1$ and $h_0<1$.
To get the last estimate, we used the Lipschitz character\footnote{We used that
\[
\begin{split}
|\delta_h \eta(x)-\delta_h\eta(y)|&\le |x-y|\,\int_0^1 \Big|\nabla \eta(x+t\,(y-x)+h)-\nabla \eta(x+t\,(y-x))\Big|\,dt\le |x-y|\, |h|\,\|D^2\eta\|_{L^\infty}.
\end{split}
\]} of $\nabla \eta$ (recall that $\eta\in C^2_0$).
\par
By combining this and \eqref{eta}, from \eqref{I10} we finally get
\begin{equation}
\label{I11}
\begin{split}
\mathcal{I}_1&\ge \frac{1}{C}\,\left[\frac{\delta_h (u\,\eta)}{|h|^\frac{\gamma+t}{p}}\right]^p_{W^{s,p}(B_R)}\\
&-C\,
\left(\frac{R}{R-r}\right)^p\,h_0^{-\gamma-t}\,\left(\left[u\right]^p_{W^{s,p}(B_{R+h_0})}+\frac{1}{(1-s)\,R^{s\,p}}\,\|u\|^p_{L^p(B_{R+h_0})}\right)\\
&-\frac{C}{(R-r)^{s\,p}}\,\left(\frac{R}{R-r}\right)^p\,\frac{1}{s\,(1-s)}\,\int_{B_R} \left|\frac{\delta_h u}{|h|^{\gamma}}\right|^p\,dx.
\end{split}
\end{equation}
\vskip.2cm\noindent
{\bf Estimate of $\mathcal{I}_2$.} By recalling that $\eta$ is supported on $B_{(R+r)/2}$, we have
\[
\begin{split}
\mathcal{I}_2&\ge -\int_{B_{(R+r)/2}}\int_{\mathbb{R}^N\setminus B_R} \Big|J_p(u_h(x)-u_h(y))-J_p(u(x)-u(y))\Big|\,\frac{|\delta_h u(x)|}{|h|^{\gamma+t}}\,\eta(x)^p\,d\mu(x,y).
\end{split}
\]
Then we observe that by basic calculus
\[
\begin{split}
\Big|J_p(u_h(x)-u_h(y))-J_p(u(x)-u(y))\Big|&\le (p-1)\,\left(|u_h(x)-u_h(y)|^{p-2}+|u(x)-u(y)|^{p-2}\right)\\
&\times |(u_h(x)-u_h(y))-(u(x)-u(y))|\\
&\le (p-1)\,\left(|u_h(x)-u_h(y)|^{p-2}+|u(x)-u(y)|^{p-2}\right)\\
&\times \Big(|\delta_h u(x)|+|\delta_h u(y)|\Big).
\end{split}
\]
Using once again the assumption \eqref{Kassummption} on the kernel $K$, we get
\begin{equation}
\label{I2}
\begin{split}
\mathcal{I}_2&\ge -C\,\int_{B_{(R+r)/2}}\left(\int_{\mathbb{R}^N\setminus B_R} \frac{|u_h(x)-u_h(y)|^{p-2}+|u(x)-u(y)|^{p-2}}{|x-y|^{N+s\,p}}\,dy\right)\left|\frac{\delta_h u(x)}{|h|^\frac{\gamma+t}{2}}\right|^2\,\eta(x)^p\,dx\\
&-C\,\int_{B_{(R+r)/2}}\left(\int_{\mathbb{R}^N\setminus B_R} \frac{|u_h(x)-u_h(y)|^{p-2}+|u(x)-u(y)|^{p-2}}{|x-y|^{N+s\,p}}\left|\frac{\delta_h u(y)}{|h|^t}\right|\,dy\right)\,\left|\frac{\delta_h u(x)}{|h|^{\gamma}}\right|\,\eta(x)^p\,dx,
\end{split}
\end{equation}
where $C=C(p,\Lambda)$. We now estimate each term on the right-hand side separately: for the first one, we have
\[
\begin{split}
\int_{B_{(R+r)/2}}&\left(\int_{\mathbb{R}^N\setminus B_R} \frac{|u_h(x)-u_h(y)|^{p-2}+|u(x)-u(y)|^{p-2}}{|x-y|^{N+s\,p}}\,dy\right)\,\left|\frac{\delta_h u(x)}{|h|^\frac{\gamma+t}{2}}\right|^2\,\eta(x)^p\,dx\\
&\le \left(\int_{B_{(R+r)/2}}\, \left(\int_{\mathbb{R}^N\setminus B_R} \frac{|u_h(x)-u_h(y)|^{p-2}+|u(x)-u(y)|^{p-2}}{|x-y|^{N+s\,p}}\,dy\right)^\frac{p}{p-2}\,dx\right)^\frac{p-2}{p}\\
&\times \left(\int_{B_{(R+r)/2}} \left|\frac{\delta_h u(x)}{|h|^\frac{\gamma+t}{2}}\right|^p\,dx\right)^\frac{2}{p}.
\end{split}
\] 
Then by Jensen's inequality\footnote{With respect to the measure $|x-y|^{-N-s\,p}\,dy$ which is finite on $\mathbb{R}^N\setminus B_{R}$, for every $x\in B_\frac{R+r}{2}$.} and with some simple manipulations, we get
\[
\begin{split}
\int_{B_\frac{R+r}{2}}&\left(\int_{\mathbb{R}^N\setminus B_R} \frac{|u_h(x)-u_h(y)|^{p-2}+|u(x)-u(y)|^{p-2}}{|x-y|^{N+s\,p}}\,dy\right)^\frac{p}{p-2}\,dx\\
&\le C\,\left(\frac{1}{s\,(R-r)^{s\,p}}\right)^\frac{2}{p-2}\,
\int_{B_\frac{R+r}{2}}\int_{\mathbb{R}^N\setminus B_R} \Big(|(\mathcal{V}_{p})_h|^2+|\mathcal{V}_{p}|^2\Big)\frac{1}{|x-y|^{N+s\,p}}\,dx\,dy,
\end{split}
\]
for some $C=C(N,p)>0$, where we recall the definition of $\mathcal{V}_p$, given in \eqref{Vp}. For the second term in the right-hand side of $\eqref{I2}$, we have
\[
\begin{split}
\int_{B_\frac{R+r}{2}}&\left(\int_{\mathbb{R}^N\setminus B_R} \frac{|u_h(x)-u_h(y)|^{p-2}+|u(x)-u(y)|^{p-2}}{|x-y|^{N+s\,p}}\,\left|\frac{\delta_h u(y)}{|h|^t}\right|\,dy\right)\,\left|\frac{\delta_h u(x)}{|h|^{\gamma}}\right|\,\eta(x)^p\,dx\\
&\le \left(\int_{B_\frac{R+r}{2}}\left(\int_{\mathbb{R}^N\setminus B_R} \frac{|u_h(x)-u_h(y)|^{p-2}+|u(x)-u(y)|^{p-2}|}{|x-y|^{N+s\,p}}\,\left|\frac{\delta_h u(y)}{|h|^t}\right|\,dy\right)^{p'}\,dx\right)^\frac{1}{p'}\\
&\times \left(\int_{B_\frac{R+r}{2}} \left|\frac{\delta_h u(x)}{|h|^{\gamma}}\right|^p\,dx\right)^\frac{1}{p}.
\end{split}
\]
By proceeding similarly as before, i.e. by using Jensen's inequality we also have
\[
\begin{split}
\int_{B_\frac{R+r}{2}}&\left(\int_{\mathbb{R}^N\setminus B_R} \frac{|u_h(x)-u_h(y)|^{p-2}+|u(x)-u(y)|^{p-2}}{|x-y|^{N+s\,p}}\,\left|\frac{\delta_h u(y)}{|h|^t}\right|\,dy\right)^{p'}\,dx\\
&\le C\,\left(\frac{1}{s\,(R-r)^{s\,p}}\right)^\frac{1}{p-1}
\int_{B_\frac{R+r}{2}}\int_{\mathbb{R}^N\setminus B_R} \frac{\Big(|u_h(x)-u_h(y)|^{p-2}+|u(x)-u(y)|^{p-2}\Big)^{p'}}{|x-y|^{N+s\,p}}\,\left|\frac{\delta_h u(y)}{|h|^t}\right|^{p'}\,dx\,dy.
\end{split}
\]
Thus from \eqref{I2} we get the following lower-bound
\begin{equation}
\label{I200}
\begin{split}
\mathcal{I}_2&\ge-C\,\left(\frac{1}{s\,(R-r)^{s\,p}}\right)^\frac{2}{p}\left(\int_{B_{\frac{R+r}{2}}}\int_{\mathbb{R}^N\setminus B_R} \Big(|(\mathcal{V}_p)_h|^{2}+|\mathcal{V}_p|^{2}\Big)\,\frac{dx\,dy}{|x-y|^{N+s\,p}}\right)^\frac{p-2}{p}\,\left(\int_{B_\frac{R+r}{2}} \left|\frac{\delta_h u}{|h|^\frac{\gamma+t}{2}}\right|^p\,dx\right)^\frac{2}{p}\\
&-C\,\left(\frac{1}{s\,(R-r)^{s\,p}}\right)^\frac{1}{p}\,\left(\int_{B_{\frac{R+r}{2}}}\int_{\mathbb{R}^N\setminus B_R} \frac{\Big(|u_h(x)-u_h(y)|^{p-2}+|u(x)-u(y)|^{p-2}\Big)^{p'}}{|x-y|^{N+s\,p}}\right.\left.\left|\frac{\delta_h u(y)}{|h|^t}\right|^{p'}\,dx\,dy\right)^\frac{1}{p'}\\
&\times\left(\int_{B_{\frac{R+r}{2}}} \left|\frac{\delta_h u}{|h|^{\gamma}}\right|^p\,dx\right)^\frac{1}{p}.
\end{split}
\end{equation}
By a further application of H\"older's inequality with exponents
\[
\frac{p}{p'}=p-1\qquad \mbox{ and } \qquad \frac{p}{p-p'}=\frac{p-1}{p-2},
\]
the second term in the right-hand side of \eqref{I200} is estimated by
\[
\begin{split}
\left(\int_{B_{\frac{R+r}{2}}}\int_{\mathbb{R}^N\setminus B_R}\right.& \left.\Big(|u_h(x)-u_h(y)|^{p-2}+|u(x)-u(y)|^{p-2}\Big)^{p'}\right.\left.\left|\frac{\delta_h u(y)}{|h|^t}\right|^{p'}\,\frac{dx\,dy}{|x-y|^{N+s\,p}}\right)^\frac{1}{p'}\\
&\le C\, \left(\int_{B_{\frac{R+r}{2}}}\int_{\mathbb{R}^N\setminus B_R} \Big(|(\mathcal{V}_p)_h|^{2}+|\mathcal{V}_p|^2\Big)\,\frac{dx\,dy}{|x-y|^{N+s\,p}}\right)^\frac{p-2}{p}\\
&\times\left(\int_{B_{\frac{R+r}{2}}}\int_{\mathbb{R}^N\setminus B_R}\left|\frac{\delta_h u(y)}{|h|^t}\right|^{p}\,\frac{dx\,dy}{|x-y|^{N+s\,p}}\right)^\frac{1}{p}.
\end{split}
\]
By using this estimate, we obtain for $\mathcal{I}_2$ the following lower bound
\begin{equation}
\label{I21}
\begin{split}
\mathcal{I}_2&\ge-C\,\left(\frac{1}{s\,(R-r)^{s\,p}}\right)^\frac{2}{p}\left(\int_{B_\frac{R+r}{2}} \left|\frac{\delta_h u}{|h|^\gamma}\right|^p\,dx\right)^\frac{2}{p}\left(\int_{B_{\frac{R+r}{2}}}\int_{\mathbb{R}^N\setminus B_R} \frac{|(\mathcal{V}_p)_h|^{2}+|\mathcal{V}_p|^{2}}{|x-y|^{N+s\,p}}\,dx\,dy\right)^\frac{p-2}{p}\\
&-C\,\left(\frac{1}{s\,(R-r)^{s\,p}}\right)^\frac{1}{p}\,\left(\int_{B_{\frac{R+r}{2}}} \left|\frac{\delta_h u}{|h|^{\gamma}}\right|^p\,dx\right)^\frac{1}{p}\left(\int_{B_{\frac{R+r}{2}}}\int_{\mathbb{R}^N\setminus B_R} \frac{|(\mathcal{V}_p)_h|^{2}+|\mathcal{V}_p|^2}{|x-y|^{N+s\,p}}\,dx\,dy\right)^\frac{p-2}{p}\\
&\times\left(\int_{B_{\frac{R+r}{2}}}\int_{\mathbb{R}^N\setminus B_R}\left|\frac{\delta_h u(y)}{|h|^t}\right|^{p}\,\frac{dx\,dy}{|x-y|^{N+s\,p}}\right)^\frac{1}{p}.
\end{split}
\end{equation}
To obtain the previous, we also observed that $|h|<h_0<1$ and used that $(\gamma+t)/2\le \gamma$.
Observe that the last term in \eqref{I21} is the integral of a nonlocal quantity containing a difference quotient $u$. By recalling Definition \ref{defi:snails} and using that $0<|h|<h_0$, we get\footnote{Observe that 
\[
\frac{1}{2}\,\mathrm{dist}\left({B_\frac{R+r}{2}},\mathbb{R}^N\setminus B_R\right)=\frac{R-r}{4}>h_0,
\]
thus we have
\[
\sup_{0<|h|<h_0}\int_{B_\frac{R+r}{2}} \mathrm{Snail}\left(\frac{\delta_h u}{|h|^t};x,B_R\right)^p\,dx\le \langle u\rangle^p_{\mathcal{Y}^{t,p}_s(B_\frac{R+r}{2};B_R)},
\]
by the very definition \eqref{snails} of the latter.}
\begin{equation}
\label{cazzo}
\begin{split}
\int_{B_{\frac{R+r}{2}}}\int_{\mathbb{R}^N\setminus B_R}\left|\frac{\delta_h u(y)}{|h|^t}\right|^{p}\,\frac{1}{|x-y|^{N+s\,p}}\,dx\,dy&=\frac{C}{R^{s\,p}}\,\int_{B_\frac{R+r}{2}} \mathrm{Snail}\left(\frac{\delta_h u}{|h|^t};x,B_R\right)^p\,dx\\
&\le\frac{C}{R^{s\,p}}\,\langle u\rangle^p_{\mathcal{Y}^{t,p}_s(B_\frac{R+r}{2};B_R)}.
\end{split}
\end{equation}
Finally, for the common $\mathcal{V}_p$ term in \eqref{I21}, we have
\[
\begin{split}
\int_{B_{\frac{R+r}{2}}}\int_{\mathbb{R}^N\setminus B_R} \Big(|(\mathcal{V}_p)_h|^{2}+|\mathcal{V}_p|^2\Big)\,\frac{dx\,dy}{|x-y|^{N+s\,p}}&\le C\, \int_{B_{\frac{R+r}{2}}}\int_{\mathbb{R}^N\setminus B_R} \frac{|u_h(x)|^p+|u_h(y)|^p}{|x-y|^{N+s\,p}}\,dx\,dy\\
&+C\,\int_{B_{\frac{R+r}{2}}}\int_{\mathbb{R}^N\setminus B_R} \frac{|u(x)|^p+|u(y)|^p}{|x-y|^{N+s\,p}}\,dx\,dy\\
&\le \frac{C}{s}\,\left(\frac{1}{R-r}\right)^{s\,p}\,\int_{B_{\frac{R+r}{2}+h_0}} |u|^p\,dx\\
&+C\,\left(\frac{1}{R}\right)^{s\,p}\,\int_{B_{\frac{R+r}{2}}} \mathrm{Snail}\left(u;x;B_{R}\right)^{p}\,dx\\
&+C\,\left(\frac{1}{R}\right)^{s\,p}\,\int_{B_{\frac{R+r}{2}}} \mathrm{Snail}\left(u_h;x;B_{R}\right)^{p}\,dx,
\end{split}
\]
for some $C=C(N,p)>0$. Moreover, by  the monotonicity properties of Snails encoded in Lemma \ref{lm:monotone}, we have
\[
\int_{B_{\frac{R+r}{2}}} \mathrm{Snail}\left(u;x;B_{R}\right)^{p}\,dx\le C\,\left(\frac{R+r}{R-r}\right)^N\,\left(\frac{R+h_0}{R-r}\right)^{s\,p}\,\langle u\rangle^p_{\mathcal{X}^p_s(B_{\frac{R+r}{2}+h_0};B_{R+h_0})}.
\]
With a simple change of variables and by observing that 
\[
B_{(R+r)/2}-h\subset B_{(R+r)/2+h_0}\qquad B_{R}-h\subset B_{R+h_0},
\]
still by Lemma \ref{lm:monotone} we get again
\[
\begin{split}
\int_{B_{\frac{R+r}{2}}} \mathrm{Snail}\left(u_h;x;B_{R}\right)^{p}\,dx&=
C\,R^{s\,p}\,\int_{B_{\frac{R+r}{2}}-h} \int_{\mathbb{R}^N\setminus (B_{R}-h)}\, \frac{|u(y)|^p}{|x-y|^{N+s\,p}}\,dx\,dy\\
&\le C\,\left(\frac{R+r}{R-r}\right)^N\,\left(\frac{R+h_0}{R-r}\right)^{s\,p}\,\langle u\rangle^p_{\mathcal{X}^p_s(B_{\frac{R+r}{2}+h_0};B_{R+h_0})}.
\end{split}
\]
By keeping everything together, we get
\begin{equation}
\label{cazzo2}
\begin{split}
\int_{B_{\frac{R+r}{2}}}\int_{\mathbb{R}^N\setminus B_R} \frac{|(\mathcal{V}_p)_h|^{2}+|\mathcal{V}_p|^2}{|x-y|^{N+s\,p}}\,dx\,dy\le \left(\frac{R+r}{R-r}\right)^N\,\left(\frac{R+h_0}{R-r}\right)^{s\,p}\,\frac{C}{s\,(R-s)^{s\,p}}\langle u\rangle^p_{\mathcal{X}^{p}_s(B_{\frac{R+r}{2}+h_0},B_{R+h_0})},
\end{split}
\end{equation}
still for $C=C(N,p)>0$.
By using \eqref{cazzo} and \eqref{cazzo2} in \eqref{I21} in conjunction with Young's inequality, we finally end up with 
\begin{equation}
\label{I210}
\begin{split}
\mathcal{I}_2&\ge-\frac{C}{(R-r)^{s\,p}}\,\frac{1}{s}\,\int_{B_\frac{R+r}{2}} \left|\frac{\delta_h u}{|h|^\frac{\gamma+t}{2}}\right|^p\,dx-\frac{C}{(R-r)^{s\,p}}\,\frac{1}{s}\,\int_{B_\frac{R+r}{2}} \left|\frac{\delta_h u}{|h|^{\gamma}}\right|^p\,dx\\
&-\left(\frac{R+r}{R-r}\right)^N\,\left(\frac{R+h_0}{R-r}\right)^{s\,p}\,\frac{C}{s\,(R-r)^{s\,p}}\,\langle u\rangle^p_{\mathcal{X}^{p}_s(B_{\frac{R+r}{2}+h_0},B_{R+h_0})}-\frac{C}{R^{s\,p}}\,\langle u\rangle^p_{\mathcal{Y}^{t,p}_s(B_\frac{R+r}{2};B_R)}.
\end{split}
\end{equation}
\vskip.2cm\noindent
{\bf Estimate of $\mathcal{I}_3$.}
This is estimated exactly in the same manner as $\mathcal{I}_2$. We thus get
\begin{equation}
\label{I31}
\begin{split}
\mathcal{I}_3&\ge-\frac{C}{(R-r)^{s\,p}}\,\frac{1}{s}\,\int_{B_\frac{R+r}{2}} \left|\frac{\delta_h u}{|h|^\frac{\gamma+t}{2}}\right|^p\,dx-\frac{C}{(R-r)^{s\,p}}\,\frac{1}{s}\,\int_{B_\frac{R+r}{2}} \left|\frac{\delta_h u}{|h|^{\gamma}}\right|^p\,dx\\
&-\left(\frac{R+r}{R-r}\right)^N\,\left(\frac{R+h_0}{R-r}\right)^{s\,p}\,\frac{C}{s\,(R-r)^{s\,p}}\langle u\rangle^p_{\mathcal{X}^{p}_s(B_{\frac{R+r}{2}+h_0},B_{R+h_0})}-\frac{C}{R^{s\,p}}\,\langle u\rangle^p_{\mathcal{Y}^{t,p}_s(B_\frac{R+r}{2};B_R)}.
\end{split}
\end{equation}
{\bf Conclusion.} From \eqref{equationh} we have
\[
\begin{split}
\mathcal{I}_1&\le |\mathcal{I}_2|+|\mathcal{I}_3|+\int_\Omega \left|\delta_h f\right|\, \left|\frac{\delta_h u}{|h|^{\gamma+t}}\right|\,\eta^p\,dx\\
&\le |\mathcal{I}_2|+|\mathcal{I}_3|+(1-s)^\frac{1}{p-1}\,R^{s\,p'}\,\int_{B_R}\left|\frac{\delta_h f}{|h|^s}\right|^{p'}\,dx+\frac{1}{(1-s)\,R^{s\,p}}\int_{B_R}\left|\frac{\delta_h u}{|h|^{\gamma+t-s}}\right|^p\,dx.
\end{split}
\]
Thus by using \eqref{I11}, \eqref{I210} and \eqref{I31} we get the conclusion, by recalling that $|h|<h_0<1$ and that
\[
\gamma+t-s\le \gamma,
\]
which follows from the hypothesis $t\le s\le \gamma$.
\end{proof}

\begin{oss}[Correction factor]
Observe that the nonlocal terms $\langle u\rangle_{\mathcal{X}^p_s}$ and $\langle u\rangle_{\mathcal{Y}^{t,p}_s}$ in the right-hand side \eqref{regularity} do not contain the correction factor $(1-s)^{-1}$, as it is natural. Indeed, if we multiply \eqref{regularity} by $(1-s)$ these terms have to disappear in the limit $s\nearrow 1$, which corresponds to the equation becoming local. 
\end{oss}

\subsection{Improving Lemma}
The proof of Theorem \ref{teo:high} is based on a combination of Proposition \ref{prop:derivatives} and of the following result. This simple result is useful in order to handle the left-hand side of \eqref{regularity}. Here second order difference quotients and Besov spaces come into play. 
\begin{lm}[Besov-Nikol'skii improvement]
\label{lm:triebel}
Let $p\ge 2$, $0<s<1$ and $0\le t\le s$. Let $B_r\Subset B_R\Subset\Omega$ be a couple of concentric balls. We take $\eta$ a standard $C^2$ cut-off function such that
\[
0\le \eta\le1,\qquad \eta\equiv 1 \quad \mbox{ on } B_r,\qquad \eta\equiv 0 \quad \mbox{ on } \mathbb{R}^N\setminus B_\frac{R+r}{2},\qquad |\nabla \eta|\le \frac{c_N}{R-r},\qquad |D^2 \eta|\le \frac{c_N}{(R-r)^2}.
\] 
Let us assume that for some $\gamma$ such that $s\le \gamma \le 1$ and some 
\[
0<h_0<\frac{1}{4}\,\min\Big\{\mathrm{dist}(B_{R};\partial\Omega),\, R-r,\,1\Big\},
\] 
we have
\[
\mathcal{M}_\gamma:=\sup_{0<|h|<h_0}\left[\frac{\delta_h (u\,\eta)}{|h|^\frac{\gamma+t}{p}}\right]^p_{W^{s,p}(B_R)}<+\infty.
\]
Then, by setting for simplicity
\begin{equation}
\label{biggamma}
\Gamma:=\frac{\gamma+t+s\,p}{p},
\end{equation}
we have the Besov-Nikol'skii estimate
\begin{equation}
\label{triebel2}
[u\,\eta]^p_{\mathcal{B}_\infty^{\Gamma,p}(\mathbb{R}^N)}\le \frac{C}{s}\,\left(\frac{R}{r}\right)^N\,\left(\frac{R}{h_0}\right)^{1+p}\,\left[(1-s)\,\mathcal{M}_\gamma+h_0^{-\Gamma\,p}\,\left\|u\right\|^p_{L^p(B_{R+h_0})}\right],
\end{equation}
for some $C=C(N,p)>0$.
In particular, we have the following estimates, for a possibly different constant $C=C(N,p)>0$:
\begin{itemize}
\item if $\boxed{\Gamma<1}$
\begin{equation}
\label{higher!}
\sup_{0<|h|<h_0}\left\|\frac{\delta_{h} u}{|h|^\Gamma}\right\|^p_{L^p(B_r)}\le \frac{C}{s\, (1-\Gamma)^p}\,\left(\frac{R}{r}\right)^N\,\left(\frac{R}{h_0}\right)^{1+p}\,\left[(1-s)\,\mathcal{M}_\gamma+h_0^{-\Gamma\,p}\, \left\|u\right\|^p_{L^p(B_{R+h_0})}\right];
\end{equation}
\vskip.2cm
\item if $\boxed{\Gamma=1}$ for every $0<\tau<1$
\begin{equation}
\label{higher!!}
\sup_{0<|h|<h_0}\left\|\frac{\delta_{h} u}{|h|^\tau}\right\|^p_{L^p(B_r)}\le \frac{C}{ (1-\tau)^p}\,\left(\frac{R}{r}\right)^N\,\left(\frac{R}{h_0}\right)^{1+p}\left[(1-s)\,\mathcal{M}_\gamma+h_0^{-p}\,\left\|u\right\|^p_{L^p(B_{R+h_0})}\right];
\end{equation}
\vskip.2cm
\item if $\boxed{\Gamma>1}$
\begin{equation}
\label{higher!!!}
\|\nabla u\|^p_{L^p(B_r)}\le \frac{C}{(\Gamma-1)^p}\,\left(\frac{R}{r}\right)^N\,\left(\frac{R}{h_0}\right)^{1+p}\left[(1-s)\,\mathcal{M}_\gamma+h_0^{-\Gamma\,p}\,\|u\|^p_{L^p(B_{R+h_0})}\right],
\end{equation}
and for every $0<\tau<\Gamma-1$
\begin{equation}
\label{higher!!!bis}
\begin{split}
\big[\nabla u\big]^p_{W^{\tau,p}(B_r)}&\le \frac{C\,(R/r)^N}{(\Gamma-1-\tau)\,\tau\,}\,\left(\frac{R}{h_0}\right)^{1+p}\left(\frac{h_0^{-\tau}}{(2-\Gamma)\,(\Gamma-1)}\right)^p\,\left[(1-s)\,\mathcal{M}_\gamma+h_0^{-\Gamma\,p}\,\|u\|^p_{L^p(B_{R+h_0})}\right].
\end{split}
\end{equation}
\end{itemize}
\end{lm}
\begin{proof}
Let $0<h<|h_0|$, by using the hypothesis and \eqref{nikolskib} with the choices
\[
\psi=\frac{\delta_h (u\,\eta)}{|h|^\frac{\gamma+t}{p}},\qquad \alpha=s,\qquad B_{\frac{R+r}{2}+h_0}\Subset B_{R}
\] 
we get
\[
\int_{\mathbb{R}^N} \left|\delta_\xi \left(\frac{\delta_h (u\,\eta)}{|h|^\frac{\gamma+t}{p}}\right)\right|^p\,\frac{1}{|\xi|^{s\,p}}\,dx\le \frac{C}{s}\,\left(\frac{R}{r}\right)^N\,\left(\frac{R}{h_0}\right)^{1+p}\,(1-s)\,\mathcal{M}_\gamma,\qquad \mbox{ for every } \xi\in\mathbb{R}^N\setminus\{0\}.
\]
Here we also used that $R-r>4\,h_0$, by hypothesis. If we now choose $\xi=h$, recall \eqref{twotimes} and take the supremum over $0<|h|<h_0$, we obtain 
\[
\sup_{0<|h|<h_0} \int_{\mathbb{R}^N} \left|\left(\frac{\delta^2_h (u\,\eta)}{|h|^\frac{\gamma+t+s\,p}{p}}\right)\right|^p\,dx\le \frac{C}{s}\,\left(\frac{R}{r}\right)^N\,\left(\frac{R}{h_0}\right)^{1+p}\,(1-s)\,\mathcal{M}_\gamma.
\]
By joining \eqref{reductionh0} and the previous estimate with simple manipulations we have
\[
\begin{split}
[u\,\eta]^p_{\mathcal{B}^{\Gamma,p}_\infty(\mathbb{R}^N)}
&\le \frac{C}{s}\,\left(\frac{R}{r}\right)^N\,\left(\frac{R}{h_0}\right)^{1+p}\,\left[(1-s)\,\mathcal{M}_\gamma+h_0^{-\Gamma\,p}\,\|u\|^p_{L^p(B_{R+h_0})}\right],
\end{split}
\]
where we used the expedient notation \eqref{biggamma}. This proves \eqref{triebel2}.
We then treat each case separately.
\vskip.2cm\noindent
{\it Case $\boxed{\Gamma<1}$}. We now use Lemma \ref{lm:not} for $\psi=u\,\eta$, then from \eqref{triebel2} we get
\begin{equation}
\label{higher!eta}
\sup_{0<|h|<h_0}\left\|\frac{\delta_{h} (u\,\eta)}{|h|^\Gamma}\right\|^p_{L^p(\mathbb{R}^N)}\le \frac{C\,(R/r)^N}{s\, (1-\Gamma)^p}\,\left(\frac{R}{h_0}\right)^{1+p}\,\left[(1-s)\,\mathcal{M}_\gamma+\left(1+h_0^{-\Gamma\,p}\right)\, \left\|u\right\|^p_{L^p(B_{R+h_0})}\right].
\end{equation}
By using the discrete Leibniz rule \eqref{leibniz}, the triangle inequality and the Lipschitz character of $\eta$, we have 
\begin{equation}
\label{overandover}
\left\|\frac{\delta_{h} u}{|h|^\Gamma}\right\|^p_{L^p(B_r)}\le C\,\left\|\frac{\delta_{h} (u\,\eta)}{|h|^\Gamma}\right\|^p_{L^p(\mathbb{R}^N)}+\frac{C}{(R-r)^p}\,h_0^{p-\Gamma\,p}\, \|u\|^p_{L^p(B_{R+h_0})},
\end{equation}
for $C=C(p)>0$.
Then \eqref{higher!} follows by using \eqref{overandover} in \eqref{higher!eta} and observing that $h_0<(R-r)$ and that $h_0<1$.
\vskip.2cm\noindent
{\it Case $\boxed{\Gamma=1}$}. Let $\tau<1$, we begin by observing that since $h_0<1$
\[
\sup_{0<|h|<h_0}\left\|\frac{\delta^2_{h} (u\,\eta)}{|h|^\tau}\right\|^p_{L^p(\mathbb{R}^N)}\le [u\,\eta]^p_{\mathcal{B}^{1,p}_\infty(\mathbb{R}^N)}.
\]
Also observe that since $\Gamma=1$, we have
\[
1=\frac{\gamma+t+s\,p}{p}\le \frac{1}{p}+s\,\frac{p+1}{p},\qquad \mbox{ i.\,e. }\ s\ge \frac{p-1}{p+1},
\]
so in this case, upon redefining the constant $C=C(N,p)>0$, we can forget the factor $1/s$ in \eqref{triebel2}. By using \eqref{reductionbis} on the left and \eqref{triebel2} on the right, we get
\[
\begin{split}
\sup_{0<|h|<h_0}\left\|\frac{\delta_{h} (u\,\eta)}{|h|^\tau}\right\|^p_{L^p(\mathbb{R}^N)}&\le \frac{C}{ (1-\tau)^p}\,\left(\frac{R}{r}\right)^N\,\left(\frac{R}{h_0}\right)^{1+p}\\ & \times \left[(1-s)\,\mathcal{M}_\gamma+\left(h_0^{-\tau\,p}+h_0^{-p}+1\right)\,\left\|u\right\|^p_{L^p(B_{R+h_0})}\right],
\end{split}
\]
possibly with a different constant $C=C(N,p)>0$. 
Finally, we use again \eqref{overandover} to remove the dependence on $\eta$ and the fact that $h_0<1$.
\vskip.2cm\noindent
{\it Case $\boxed{\Gamma>1}$}. We first observe that due to the restrictions on the parameters, we always have $\Gamma <2$. Moreover, as in the previous case we still have $s\ge (p-1)/(p+1)$, thus again we can forget the factor $1/s$ in \eqref{triebel2}.
Then we use Proposition \ref{prop:yes} with $\psi=u\,\eta$ and from \eqref{triebel2} we get
\[
\begin{split}
\|\nabla (u\,\eta)\|^p_{L^p(\mathbb{R}^N)}&\le C\,\left(\frac{R}{r}\right)^N\,\left(\frac{R}{h_0}\right)^{1+p}\,\left[\frac{(1-s)\,\mathcal{M}_\gamma}{(\Gamma-1)^p}+\left(1+\frac{h_0^{-\Gamma\,p}}{(\Gamma-1)^p}\right)\,\|u\|^p_{L^p(B_{R+h_0})}\right].
\end{split}
\]
By recalling that $\eta\equiv 1$ on $B_r$ and observing that $(\Gamma-1)<1$, we get \eqref{higher!!!}.
As for \eqref{higher!!!bis}, we observe that still by Proposition \ref{prop:yes} and \eqref{triebel2} we also have
\[
\begin{split}
\sup_{0<|h|<h_0}\left\|\frac{\delta_{h} \nabla (u\,\eta)}{|h|^{\Gamma-1}}\right\|^p_{L^p(\mathbb{R}^N)}\le \frac{C}{(2-\Gamma)^p\,(\Gamma-1)^p}\left(\frac{R}{r}\right)^N\,\left(\frac{R}{h_0}\right)^{1+p}\,\left[(1-s)\,\mathcal{M}_\gamma+h_0^{-\Gamma\,p}\,\|u\|^p_{L^p(B_{R+h_0})}\right].
\end{split}
\]
If we now apply Proposition \ref{prop:lostinpassing} to the compactly supported function $\psi=\nabla(u\,\eta)$ and the exponent $\beta=\Gamma-1$ we get
\[
\begin{split}
[\nabla u]^p_{W^{\tau,p}(B_r)}&\le [\nabla (u\,\eta)]^p_{W^{\tau,p}(\mathbb{R}^N)}\\
&\le C\,\left(\frac{h_0^{(\Gamma-1-\tau)\,p}}{\Gamma-1-\tau}\, \sup_{0<|h|<h_0} \left\|\frac{\delta_h \nabla (u\,\eta)}{|h|^{\Gamma-1}}\right\|_{L^p(\mathbb{R}^N)}^p+\frac{h_0^{-\tau\,p}}{\tau}\,\|\nabla (u\,\eta)\|^p_{L^p(\mathbb{R}^N)}\right),
\end{split}
\]
for every $0<\tau<\Gamma-1$.
The right-hand side is now estimated by appealing to the previous two estimates, thus we conclude the proof with standard manipulations.
\end{proof}

\section{Proof of Theorem \ref{teo:high}}
\label{sec:proof}

Let $R>0$ and $B_R\Subset \Omega$, we want to prove the estimates \eqref{screambloodygore} and \eqref{leprosy}-\eqref{spiritualhealing} on the balls $B_{R/2}$ and $B_{R/4}$.
Without loss of generality, we can assume that $B_R$ is centered at the origin.
Observe that if we consider the rescaled functions
\[
u_R(x)=u(R\,x)\qquad \mbox{ and }\qquad f_R(x)=R^{s\,p}\,f(R\,x),\qquad x\in R^{-1}\,\Omega,
\]
then $u_R\in W^{s,p}_{loc}(R^{-1}\,\Omega)\cap \mathcal{Y}^{t,p}_s$ is a local weak solution in the rescaled set $R^{-1}\,\Omega$, with right-hand side $f_R$.
Thus we just need to estimate
\[
[u_R]_{W^{\tau,p}(B_{1/2})},\qquad \mbox{ for every } \tau<\frac{t+s\,p}{p-1},
\]
or
\[
\|\nabla u_R\|_{L^p(B_{1/2})}+[\nabla u_R]_{W^{\tau,p}(B_{1/4})},\qquad\mbox{ for every } \tau<\frac{t+s\,p}{p}-\frac{p-1}{p}.
\]
The desired results will be then obtained by scaling back. 
\vskip.2cm\noindent
\subsection{The general scheme} As explained in the Introduction, the desired estimates are proved by means on an iterative scheme. First of all, we define the sequence
\begin{equation}
\label{exponents}
\gamma_0=s,\qquad \gamma_{i+1}=\frac{\gamma_i+t+s\,p}{p}.
\end{equation}
We observe $\gamma_i$ is strictly increasing and
\begin{equation}
\label{limit}
\lim_{i\to\infty} \gamma_i=\frac{t+s\,p}{p-1}.
\end{equation}
We take any index $i_0\in\mathbb{N}\setminus\{0\}$ such that 
\[
\gamma_{i_0-1}<1, 
\]
the precise choice of $i_0$ will be done below. We define the decreasing sequence of radii
\[
r_i:=\frac{3}{4}-\frac{i}{i_0}\,\frac{1}{4},\qquad i\in\{0,\dots,i_{0}\}.
\]
Accordingly, we consider the concentric balls $B_{r_i}$ and observe that 
\[
B_{r_0}=B_{3/4}\qquad \mbox{ and }\qquad B_{r_{i_0}}=B_{1/2}.
\] 
We point out that by construction, we have
\begin{equation}
\label{radii}
r_i-r_{i+1}=\frac{1}{4\,i_0},\qquad \frac{1}{2}\le r_i<1,\qquad \mbox{ and }\qquad\frac{r_i}{r_i-r_{i+1}}<4\,i_0.
\end{equation}
Then we define
\begin{equation}
\label{h0}
h_0=\frac{1}{100\,i_0},
\end{equation}
thus with such a choice we have
\[
h_0<\frac{1}{4}\,\min\Big\{\mathrm{dist}\Big(B_{r_i},\partial(R^{-1}\,\Omega)\Big),r_{i}-r_{i+1},1\Big\},\qquad i=0,\dots,i_0-1.
\]
Finally, for every $i\in\{0,\dots,i_0-1\}$ we choose a standard cut-off function $\eta_i\in C^2_0(B_{r_i})$ such that
\[
\begin{split}
0\le \eta_i\le 1,&\qquad \eta_i\equiv 1 \quad \mbox{ on } B_{r_{i+1}},\quad 
 \eta_i\equiv 0 \quad \mbox{ on } \mathbb{R}^N\setminus B_\frac{r_i+r_{i+1}}{2},\\ 
& |\nabla \eta_i|\le \frac{c_N}{r_{i}-r_{i+1}}=4\,c_N\,i_0 \qquad \mbox{ and }\qquad |D^2\eta|\leq \frac{c_N}{(r_i-r_{i+1})^2}=16\,c_N\,i_0^2.
\end{split}
\] 
By taking into account \eqref{radii} and \eqref{h0}, for every $0<|h|<h_0$ by Proposition \ref{prop:derivatives} with simple manipulations we get (recall the definition of $h_0$ and that $i_0\ge 1$)
\begin{equation}
\label{prop31i}
\begin{split}
\left[\frac{\delta_h (u_R\,\eta_i)}{|h|^\frac{\gamma_i+t}{p}}\right]^p_{W^{s,p}(B_{r_i})}&\le \frac{C}{(1-s)\,s}\,i_0^{2\,p}\,\left\|\frac{\delta_h u_R}{|h|^{\gamma_i}}\right\|^p_{L^p(B_{r_i})}\\
&+C\,i_0^{p+\gamma_i+t}\,\left(\left[u_R\right]^p_{W^{s,p}(B_{r_i+h_0})}+\frac{1}{s\,(1-s)}\,\|u_R\|^p_{L^p(B_{r_i+h_0})}\right)\\
&+C\,(1-s)^\frac{1}{p-1}\,\left\|\frac{\delta_h f_R}{|h|^s}\right\|^{p'}_{L^{p'}(B_{r_i})}\\
&+\frac{C}{s}\,i_0^{N+2\,p}\,\langle u_R\rangle^p_{\mathcal{X}^{p}_s(B_{\frac{r_i+r_{i+1}}{2}+h_0},B_{r_i+h_0})}\\
&+C\,\langle u_R\rangle^p_{\mathcal{Y}^{t,p}_s(B_\frac{r_i+r_{i+1}}{2};B_{r_i})},\qquad\qquad i=0,\dots,i_0-1.
\end{split}
\end{equation}
for some $C=C(N,p)>0$. Before going on, we try to simplify the previous estimate.
\par
By construction $B_{r_i+h_0}\subset B_1$ for every $i=0,\dots,i_0$, then by Proposition \ref{lm:nikolski} ({\it local case}) we get
\begin{equation}
\label{prop31f}
(1-s)^\frac{1}{p-1}\,\sup_{0<|h|<h_0} \left\|\frac{\delta_h f_R}{|h|^{s}}\right\|_{L^{p'}(B_{r_i})}^{p'}\le \frac{C}{h_0^{p'}}\,(1-s)^{p'}\,\left[[f_R]^{p'}_{W^{s,p'}(B_1)}+\frac{1}{s\,(1-s)}\,\|f_R\|^{p'}_{L^{p'}(B_1)}\right],
\end{equation}
where we used again that $h_0<1$. Also, by the monotonicity properties of Lemma \ref{lm:monotone}
\begin{equation}
\label{prop31snail}
\langle u_R\rangle^p_{\mathcal{X}^{p}_s(B_{\frac{r_i+r_{i+1}}{2}+h_0},B_{r_i+h_0})}\le C\,i_0^{N+p}\,\langle u_R\rangle_{\mathcal{X}^{p}_s(B_{\frac{3}{4}};B_1)}^p,
\end{equation}
where we used that by construction
\[
B_{\frac{r_i+r_{i+1}}{2}+h_0}\subset B_\frac{3}{4},\qquad i=0,\dots,i_0.
\]
Finally, by observing that 
\[
\frac{1}{2}\, \mathrm{dist}\left(B_\frac{r_i+r_{i+1}}{2},\mathbb{R}^N\setminus B_{r_i}\right)=\frac{r_i-r_{i+1}}{4}=\frac{1}{16\,i_0}\le\frac{1}{16}=\frac{1}{2}\,\mathrm{dist}\left(B_{\frac{3}{4}},\mathbb{R}^N\setminus B_\frac{7}{8}\right),
\]
by \eqref{snailing0} with the choices 
\[
F_1=B_\frac{r_i+r_{i+1}}{2},\qquad E_1=B_{r_i},\qquad F_2=B_\frac{3}{4},\qquad E_2=B_\frac{7}{8},
\]
we get
\begin{equation}
\label{prop31snailors!}
\begin{split}
\langle u_R\rangle^p_{\mathcal{Y}^{t,p}_s(B_\frac{r_i+r_{i+1}}{2};B_{r_i})}&=\sup_{0<|h|<\frac{r_i-r_{i+1}}{4}}\,\int_{B_\frac{r_i+r_{i+1}}{2}} \mathrm{Snail}\left(\frac{\delta_h u_R}{|h|^t};x,B_{r_i}\right)^p\,dx\\
&\le\langle u_R\rangle^p_{\mathcal{Y}^{t,p}_s(B_{\frac{3}{4}};B_\frac{7}{8})}+C\,i_0^{N+p}\,\sup_{0<|h|<\frac{1}{16}}\,\left\|\frac{\delta_h u_R}{|h|^{t}}\right\|_{L^p(B_\frac{7}{8})}^p,
\end{split}
\end{equation}
for some $C=C(N,p)>0$. The last local term can be further estimated by Proposition \ref{lm:nikolski} ({\it local case}) as follows (recall that $t\le s$)
\begin{equation}
\label{prop31snailors!bis}
\sup_{0<|h|<\frac{1}{16}} \left\|\frac{\delta_h u_R}{|h|^{t}}\right\|_{L^p(B_\frac78)}^p\le C\,(1-s)\left[[u_R]^p_{W^{s,p}(B_{1})}+\frac{1}{s\,(1-s)}\,\|u_R\|^p_{L^p(B_1)}\right].
\end{equation}
By using \eqref{prop31f}, \eqref{prop31snail}, \eqref{prop31snailors!} and \eqref{prop31snailors!bis} in \eqref{prop31i} and observing that
\[
h_0^{-p'}\le C\, i_0^{4\,(N+p)},\qquad i_0^{p+\gamma_{i}+t}\le C\, i_0^{4\,(N+p)},\qquad i_0^{N+p}\le i_0^{4\,(N+p)},
\]
for every $0<|h|<h_0$ we obtain
\begin{equation}
\label{scheme}
\begin{split}
\left[\frac{\delta_h (u_R\,\eta_i)}{|h|^\frac{\gamma_i+t}{p}}\right]^p_{W^{s,p}(B_{r_i})}&\le \frac{C}{(1-s)\,s}\,i_0^{2\,p}\,\left\|\frac{\delta_h u_R}{|h|^{\gamma_i}}\right\|^p_{L^p(B_{r_i})}+C\, i_0^{4\,(N+p)}\,\mathcal{A}_1(u_R,f_R),\qquad i=0,\dots,i_0-1,
\end{split}
\end{equation}
where $\mathcal{A}_1$ is the quantity defined in \eqref{AR}. In what follows, for simplicity we just write $\mathcal{A}_1$ in place of $\mathcal{A}_1(u_R,f_R)$. 
Observe that $\mathcal{A}_1<+\infty$, thanks to the assumptions on $u$ and $f$. 
\par
We now set
\[
\mathcal{M}_{\gamma_i}:=\sup_{0<|h|<h_0}\left[\frac{\delta_h (u_R\,\eta_i)}{|h|^\frac{\gamma_i+t}{p}}\right]^p_{W^{s,p}(B_{r_i})},\qquad i=0,\dots,i_0-1,
\]
and claim that 
\begin{equation}
\label{Mi}
\mathcal{M}_{\gamma_i}<+\infty,\qquad\mbox{ for every }i=0,\dots,i_0-1.
\end{equation}
This is true by a finite induction: for $i=0$, we have $\gamma_0=s$ and by combining \eqref{scheme} and Proposition \ref{lm:nikolski} ({\it local case}) we get
\[
\mathcal{M}_{\gamma_0}\le \frac{C}{s}\,\frac{i_0^{2\,p}}{h_0^p},\left[[u_R]^p_{W^{s,p}(B_{1})}+\frac{1}{s\,(1-s)}\,\|u_R\|^p_{L^p(B_1)}\right]+C\,i_0^{4\,(N+p)}\,\mathcal{A}_1,
\]
where we used again that $B_{r_0+h_0}\subset B_1$. Thus the claim is true for $i=0$. Also, 
by using the definition of $\mathcal{A}_1$ and \eqref{h0}, we can infer
\begin{equation}
\label{step0}
\mathcal{M}_{\gamma_0}\le \frac{C_0}{s}\,i_0^{4\,(N+p)}\,\mathcal{A}_1,
\end{equation}
where as usual $C_0=C_0(N,p)>0$.
\par
Let us now assume that $\mathcal{M}_{\gamma_i}<+\infty$ for an index\footnote{Of course, if $i_0=1$ there is nothing to prove.} $i\in\{0,\dots,i_0-2\}$, then we can use Lemma \ref{lm:triebel}.
Namely, by combining \eqref{higher!} and \eqref{scheme} we get 
\[
\mathcal{M}_{\gamma_{i+1}}\le \frac{C\,i_0^{2\,p}\,h_0^{-1-p}}{s^2\,(1-\gamma_{i+1})^p}\,\mathcal{M}_{\gamma_i}+\frac{C\,i_0^{2\,p}\,h_0^{-1-p}}{s\,(1-\gamma_{i+1})^p}\,h_0^{-\gamma_{i+1}\,p}\,\frac{1}{s\,(1-s)}\,\|u_R\|^p_{L^p(B_1)}+C\,i_0^{4\,(N+p)}\,\mathcal{A}_1,
\]
where $C=C(N,p)>0$ is a possibly different constant and we used the relation between $\gamma_i$ and $\gamma_{i+1}$ and the fact that $\gamma_{i_0-1}<1$ . This in turn shows that $\mathcal{M}_{\gamma_{i+1}}<+\infty$ and thus the validity of \eqref{Mi}.
\par
As before, at first we try to simplify the previous estimate. Observe that
\[
\begin{split}
h_0^{-\gamma_{i+1}\,p-p-1}\,\frac{i_0^{2\,p}}{s\,(1-s)}\,\|u_R\|^p_{L^p(B_1)}
&\le C\,i_0^{4\,p+1}\,\mathcal{A}_1,\qquad i=0,\dots,i_0-2,
\end{split}
\]
where we used the definition \eqref{h0} of $h_0$ and the fact that $\gamma_{i+1}\le \gamma_{i_0-1}<1$. From the previous discussion and \eqref{step0} we thus obtain the iterative scheme 
\begin{equation}
\label{ready?}
\left\{\begin{array}{lc}
\mathcal{M}_{\gamma_0}\le \dfrac{C_0}{s}\,i_0^{4\,(N+p)}\,\mathcal{A}_1,&\\
&\\
\mathcal{M}_{\gamma_{i+1}}\le \dfrac{C_1\,i_0^{4\,p}}{s^2\,(1-\gamma_{i+1})^p}\,\mathcal{M}_{\gamma_i}+\dfrac{C_2\,i_0^{4\,(N+p)}}{s\,(1-\gamma_{i+1})^p}\,\mathcal{A}_1,& \mbox{ for } i=0,\dots,i_0-2,
\end{array}
\right.
\end{equation}
where $C_1=C_1(N,p)>0$ and $C_2=C_2(N,p)>0$. It is intended that the second estimate in \eqref{ready?} is void when $i_0=1$. Without loss of generality, we can assume that $C_1\ge 1$.
\vskip.2cm\noindent
\subsection{Case $t+s\,p\le (p-1)$}
We fix a differentiability exponent $\tau$ such that
\[
s\le \tau<\frac{t+s\,p}{p-1},
\] 
as in \eqref{screambloodygore}, then the index $i_0\in\mathbb{N}\setminus\{0\}$ above is chosen so that 
\[
\tau<\gamma_{i_0}<1.
\]
This is possible thanks to \eqref{limit}. We recall that $\gamma_{i+1}\le \gamma_{i_0-1}<1$ for every $i=0,\dots,i_0-2$. By using this observation in \eqref{ready?} and iterating, we get
\begin{equation}
\label{estimate}
\begin{split}
\mathcal{M}_{\gamma_{i_0-1}}&\le \left(\frac{C_1\,i_0^{4\,p}}{s^2\,(1-\gamma_{i_0-1})^p}\right)^{i_0-1}\,\frac{C_0}{s}\,i_0^{4\,(N+p)}\,\mathcal{A}_1\\
&+\left[\sum_{i=0}^{i_0-2}\left(\frac{C_1\,i_0^{4\,p}}{s^2\,(1-\gamma_{i_0-1})^p}\right)^i\right]\,\left[\frac{C_2\,i_0^{4\,(N+p)}}{s\,(1-\gamma_{i_0-1})^p}\,\mathcal{A}_1\right]\le \left(\frac{C_3}{s^2}\,\frac{i_0^{4\,(N+p)}}{(1-\gamma_{i_0-1})^p}\right)^{i_0}\,\mathcal{A}_1,
\end{split}
\end{equation}
where $C_3=\max\{C_0,\,C_1,\,C_2\}\ge 1$, since $C_1\ge 1$.
\par
We are ready to perform the final step. We use again Lemma \ref{lm:triebel}, then \eqref{triebel2} and \eqref{estimate} yield the Besov-Nikol'skii estimate
\[
[u_R\,\eta_{i_0-1}]^p_{\mathcal{B}^{\gamma_{i_0},p}_\infty(\mathbb{R}^N)}\le \frac{C i_0^{2\,p}}{s}\,\left[\left(\frac{C_3}{s^2}\,\frac{i_0^{4\,(N+p)}}{(1-\gamma_{i_0})^p}\right)^{i_0}\,(1-s)\,\mathcal{A}_1+h_0^{-\gamma_{i_0}\,p}\,\left\|u_R\right\|^p_{L^p(B_1)}\right],
\]
where we further used that $0<1-\gamma_{i_0}<1-\gamma_{i_0-1}$.
The left-hand side is estimated from below thanks to \eqref{reduction}, thus we get
\[
\begin{split}
\sup_{0<|h|<h_0}\left\|\frac{\delta_{h} (u_R\,\eta_{i_0-1})}{|h|^{\gamma_{i_0}}}\right\|^p_{L^p(\mathbb{R}^N)}&\le \frac{Ci_0^{2p}}{s(1-\gamma_{i_0})^p}\,\left[\left(\frac{C_3}{s^2}\,\frac{i_0^{4\,(N+p)}}{(1-\gamma_{i_0})^p}\right)^{i_0}\,(1-s)\,\mathcal{A}_1+h_0^{-\gamma_{i_0}\,p}\,\left\|u_R\right\|^p_{L^p(B_1)}\right].
\end{split}
\]
We now recall that $i_0$ has been chosen so that $\gamma_{i_0}> \tau$, by applying Proposition \ref{prop:lostinpassing} we get
\[
\begin{split}
[u_R\,\eta_{i_0-1}]^p_{W^{\tau,p}(\mathbb{R}^N)}&\le \frac{h_0^{(\gamma_{i_0}-\tau)\,p}}{\gamma_{i_0}-\tau}\,\frac{Ci_0^{2p}}{s(1-\gamma_{i_0})^p}\,\left[\left(\frac{C_3}{s^2}\,\frac{i_0^{4\,(N+p)}}{(1-\gamma_{i_0})^p}\right)^{i_0}\,(1-s)\,\mathcal{A}_1+h_0^{-\gamma_{i_0}\,p}\,\left\|u_R\right\|^p_{L^p(B_1)}\right]\\
&+\frac{C\,h_0^{-\tau\,p}}{\tau}\,\|u_R\|_{L^p(B_1)}^p.
\end{split}
\]
On the other hand $\eta_{i_0-1}\equiv 1$ on $B_{r_{i_0}}=B_{1/2}$ and by definition of $h_0$ and the fact that $\tau \ge s$
\[
\frac{h_0^{-\tau\,p}}{\tau}\,\|u_R\|_{L^p(B_1)}^p\le \frac{h_0^{-\gamma_{i_0}\,p}}{\tau}\,\left\|u_R\right\|^p_{L^p(B_1)}\le C\,(1-s)\,i_0^p\,\mathcal{A}_1.
\]
Thus we conclude with the estimate (we use that $\tau\ge s$ and again $h_0<1$)
\begin{equation}
\label{gore}
[u_R]^p_{W^{\tau,p}(B_{1/2})}\le \frac{1}{(\gamma_{i_0}-\tau)}\,\left(\frac{C_4}{s^2}\,\frac{i_0^{4\,(N+p)}}{(1-\gamma_{i_0})^p}\right)^{i_0+1}\,(1-s)\,\mathcal{A}_1,
\end{equation}
where $C_4\ge 1$ as usual depends on $N$ and $p$ only. We now scale back in order to catch the desired estimate for $u$ in $B_{R/2}$. By recalling the definition \eqref{AR} of $\mathcal{A}_1$, from \eqref{gore} with a simple change of variables we exactly get \eqref{screambloodygore}. The constant $\mathcal{C}_1$ appearing in \eqref{screambloodygore} is given by
\[
\mathcal{C}_1:=\frac{1}{(\gamma_{i_0}-\tau)}\,\left(\frac{C_4}{s^2}\,\frac{i_0^{4\,(N+p)}}{(1-\gamma_{i_0})^p}\right)^{i_0+1}.
\]
\subsection{Case $t+s\,p>(p-1)$} We first point out that in this case we have
\begin{equation}
\label{frombelow}
s>\frac{p-1}{p+1}.
\end{equation}
We still consider the sequence $\{\gamma_i\}_{i\in\mathbb{N}}$ defined by \eqref{exponents}.
Observe that in this case
\[
\lim_{i\to\infty} \gamma_i=\frac{t+s\,p}{p-1}>1.
\]
Then this time the index $i_0\in\mathbb{N}\setminus\{0\}$ is chosen so that 
\[
\gamma_{i_0-1}< 1 \qquad \mbox{ and }\qquad \gamma_{i_0}=\frac{\gamma_{i_0-1}+t+s\,p}{p}\ge 1,
\]
which is feasible. From the scheme \eqref{ready?}, by using that $\gamma_{i+1}<\gamma_{i_0-1}<1$ for $i=0,\dots,i_0-2$, we get
\begin{equation}
\label{ready?2}
\mathcal{M}_{\gamma_{i_0-1}}\le \left(\frac{C_3'\, i_0^{4\,(N+p)}}{(1-\gamma_{i_0-1})^p}\right)^{i_0}\,\mathcal{A}_1,
\end{equation}
exactly as in \eqref{estimate}, with $C_3'=C_3'(N,p)\ge 1$. We also used \eqref{frombelow} in order to rule out the factor $1/s^2$.
\par\noindent
We need to make a distinction between two possible subcases
\[
\gamma_{i_0}>1\qquad \mbox{ or }\qquad \gamma_{i_0}=1.
\]
\vskip.2cm\noindent
{\tt Case $\boxed{\gamma_{i_0}>1}$.}
Since $\gamma_{i_0}>1$, we can apply \eqref{higher!!!} of Lemma \ref{lm:triebel} and get 
\[
\begin{split}
\|\nabla u_R\|^p_{L^p(B_{1/2})}&\le \frac{Ci_0^{2p}}{(\gamma_{i_0}-1)^p}\,\left[(1-s)\,\mathcal{M}_{\gamma_{i_0-1}}+h_0^{-p\,\gamma_{i_0}}\,\|u_R\|^p_{L^p(B_1)}\right],
\end{split}
\]
which shows that $\|\nabla u_R\|_{L^p(B_{1/2})}<+\infty$. By using \eqref{ready?2} in the previous estimate and the definitions of $h_0$ and of $\mathcal{A}_1$, we end up with
\begin{equation}
\label{gradient}
\begin{split}
\|\nabla u_R\|^p_{L^p(B_{1/2})}&\le\frac{C}{(\gamma_{i_0}-1)^p}\,\left(C_5\,\frac{i_0^{4\,(N+p)}}{(1-\gamma_{i_0-1})^p}\right)^{i_0+1}\,(1-s)\,\mathcal{A}_1,
\end{split}
\end{equation}
where $C_5=C_5(N,p)\ge 1$. 
By going back to the original solution $u$ with a scaling, we get \eqref{leprosy} with the constant $\mathcal{C}_2$ given by
\[
\mathcal{C}_2:=\frac{C}{(\gamma_{i_0}-1)^p}\,\left(C_5\,\frac{i_0^{4\,(N+p)}}{(1-\gamma_{i_0-1})^p}\right)^{i_0+1}.
\]
We still need to prove the fractional differentiability of the gradient. Observe that if we directly apply estimate \eqref{higher!!!bis} of Lemma \ref{lm:triebel} with $\gamma_{i_0-1}$, we would get the weaker result\footnote{Indeed, observe that $\gamma_{i_0}<(1+t+s\,p)/p$.}
\[
[\nabla u_R]_{W^{\tau,p}(B_{1/2})}<+\infty,\qquad\mbox{ for every } \tau<\gamma_{i_0}-1.
\]
Thus, we have to proceed differently. First of all, we introduce the new cut-off function 
$\widetilde \eta\in C^2_0(B_{3/8})$ such that
\[
\begin{split}
0\le \widetilde \eta\le 1,&\qquad \widetilde \eta\equiv 1 \quad \mbox{ on } B_{1/4},\qquad 
 \widetilde \eta\equiv 0 \quad \mbox{ on } \mathbb{R}^N\setminus B_{5/16},\\ 
& |\nabla \widetilde\eta|\le c_N' \qquad \mbox{ and }\qquad |D^2\widetilde\eta|\leq c_N'.
\end{split}
\]
Then we observe that since $u_R\in W^{1,p}(B_{1/2})$ we have
\begin{equation}
\label{germano}
\sup_{0<|h|<h_0}\left\|\frac{\delta_h u_R}{|h|}\right\|^p_{L^p(B_{3/8})}\le C\,\|\nabla u_R\|^p_{L^p(B_{1/2})}.
\end{equation}
We can now use Proposition \ref{prop:derivatives} in the limit case $\gamma=1$ and with balls $B_{1/4}\Subset B_{3/8}$, this gives
\[
\begin{split}
\widetilde{\mathcal{M}}_1&:=\sup_{0<|h|<h_0}\left[\frac{\delta_h (u_R\,\widetilde\eta)}{|h|^\frac{1+t}{p}}\right]^p_{W^{s,p}(B_{3/8})}\\
&\le \frac{C}{(1-s)\,s}\,i_0^{2\,p}\,\sup_{0<|h|<h_0}\left\|\frac{\delta_h u_R}{|h|}\right\|^p_{L^p(B_{3/8})}+C\,i_0^{4\,(N+p)}\,\mathcal{A}_1.
\end{split}
\]
By combining \eqref{germano} and \eqref{gradient} and still using \eqref{frombelow}, $\widetilde{\mathcal{M}}_1$ can be further estimated by
\[
\widetilde{\mathcal{M}}_1\le  C\,i_0^{2\,p}\,\left(C_5\,\frac{i_0^{4\,(N+p)}}{(1-\gamma_{i_0-1})^p}\right)^{i_0+1}\,\frac{1}{(\gamma_{i_0}-1)^p}\,\mathcal{A}_1.
\]
Then by using estimate \eqref{higher!!!bis} of Lemma \ref{lm:triebel} for $\gamma=1$ and the previous inequality for $\widetilde{\mathcal{M}}_1$, we get 
\[
\begin{split}
[\nabla u_R]^p_{W^{\tau,p}(B_{1/4})}&\le\frac{h_0^{-\tau\,p}}{(\Gamma-1-\tau)\,\tau}\, \frac{Ci_0^{2\,p}}{(2-\Gamma)^p\,(\Gamma-1)^p}\,\left[C\,i_0^{2\,p}\,\left(C_5\,\frac{i_0^{4(N+p)}}{(1-\gamma_{i_0-1})^p}\right)^{i_0+1}\,\frac{1-s}{(\gamma_{i_0}-1)^p}\,\mathcal{A}_1\right.\\
&+\left.h_0^{-\Gamma\,p}\,\|u_R\|^p_{L^p(B_1)}\right],
\end{split}
\] 
where $\Gamma=(1+t+s\,p)/p$. The usual elementary manipulations used so far then give
\[
[\nabla u_R]^p_{W^{\tau,p}(B_{1/4})}\le\frac{(2-\Gamma)^{-p}\,(\Gamma-1)^{-p}}{(\Gamma-1-\tau)\,\tau}\,\left(C_7\,\frac{i_0^{4\,(N+p)}}{(1-\gamma_{i_0-1})^p}\right)^{i_0+2}\,\frac{1-s}{(\gamma_{i_0}-1)^p}\,\mathcal{A}_1.
\]
By scaling we get \eqref{spiritualhealing} as desired, with the constant $\mathcal{C}_3$ given by
\[
\mathcal{C}_3=\left(C_7\,\frac{i_0^{4\,(N+p)}}{(1-\gamma_{i_0-1})^p}\right)^{i_0+2}\,\frac{1}{(\gamma_{i_0}-1)^p},
\]
and $C_7>0$ depending on $N$ and $p$ only.
This concludes the proof in the subcase $\gamma_{i_0}>1$.
\vskip.2cm\noindent
{\tt Case $\boxed{\gamma_{i_0}=1}$.} This case is subtle, due to the fact that $\mathcal{B}^{1,p}_\infty\not\subset W^{1,p}$. Rather than jumping directly from the ball $B_{r_{i_0-1}}$ to the final one $B_{r_{i_0}}$ as before, we need to slightly ``rectify'' the scheme. 
\par
First of all, we introduce the new intermediate ball $\widetilde B$
\[
B_{r_{i_0}}\Subset \widetilde B:=B_\frac{r_{i_0-1}+3\,r_{i_0}}{4}\Subset  B_\frac{r_{i_0}+r_{i_0-1}}{2}\Subset B_{r_{i_0-1}}.
\]
Then we replace the cut-off function $\eta_{i_0-1}\in C^2_0(B_{r_{i_0}})$ with the new one $\widetilde \eta$ such that
\[
\begin{split}
0\le \widetilde \eta\le 1,&\qquad \widetilde \eta\equiv 1 \quad \mbox{ on } \widetilde B,\qquad 
 \widetilde \eta\equiv 0 \quad \mbox{ on } \mathbb{R}^N\setminus B_\frac{r_{i_0}+r_{i_0-1}}{2},\\ 
& |\nabla \widetilde\eta|\le \frac{c'_N}{r_{i_0-1}-r_{i_0}} \qquad \mbox{ and }\qquad |D^2\widetilde\eta|\leq \frac{c'_N}{(r_{i_0-1}-r_{i_0})^2}.
\end{split}
\]
Finally, we set
\[
\widetilde{\mathcal{M}}_{\gamma_{i_0-1}}:=\sup_{0<|h|<h_0}\left[\frac{\delta_h (u_R\,\widetilde\eta)}{|h|^\frac{\gamma_{i_0-1}+t}{p}}\right]^p_{W^{s,p}(B_{r_{i_0-1}})}.
\]
We now proceed by iteration as in the proof of \eqref{ready?2}, but in the last step we replace \eqref{ready?} with
\[
\widetilde{\mathcal{M}}_{\gamma_{i_0-1}}\le \dfrac{C_1\,i_0^{4\,p}}{(1-\gamma_{i_0-1})^p}\,\mathcal{M}_{\gamma_{i_0-2}}+\dfrac{C_2\,i_0^{4\,(N+p)}}{(1-\gamma_{i_0-1})^p}\,\mathcal{A}_1.
\]
The latter can be proved as before by combining \eqref{higher!} and \eqref{scheme}.
Thus this time we get
\[
\widetilde{\mathcal{M}}_{\gamma_{i_0-1}}\le \left(C_3'\,\frac{i_0^{4\,(N+p)}}{(1-\gamma_{i_0-1})^p}\right)^{i_0}\,\mathcal{A}_1.
\]
An application of estimate \eqref{higher!!} of Lemma \ref{lm:triebel} gives\footnote{Observe that by construction the difference of the radii of the two balls $\widetilde B$ and $B_{r_{i_0-1}}$ is such that
\[
r_{i_0-1}-\frac{r_{i_0-1}+3\,r_{i_0}}{4}=\frac{3}{4}\,(r_{i_0-1}-r_{i_0})=\frac{3}{16\,i_0}>4\,h_0.
\]}
\[
\sup_{0<|h|<h_0}\left\|\frac{\delta_{h} u_R}{|h|^\beta}\right\|^p_{L^p(\widetilde B)}\le \frac{C\,i_0^{2\,p}}{(1-\beta)^p}\,\left[(1-s)\,\widetilde{\mathcal{M}}_{\gamma_{i_0-1}}+h_0^{-p}\,\left\|u_R\right\|^p_{L^p(B_{1})}\right],
\]
for an arbitrary $0<\beta<1$ (we still used \eqref{frombelow} to neglect the factor $1/s$). We now apply Proposition \ref{prop:derivatives} with balls $B_{r_{i_0}}\Subset \widetilde B$, this would give as at the beginning of the proof
\[
\left[\frac{\delta_h (u_R\,\eta)}{|h|^\frac{\beta+t}{p}}\right]^p_{W^{s,p}(\widetilde B)}\le \frac{C}{(1-s)\,s}\,i_0^{2\,p}\,\left\|\frac{\delta_h u_R}{|h|^{\beta}}\right\|^p_{L^p(\widetilde B)}+C\, i_0^{4\,(N+p)}\,\mathcal{A}_1,\qquad 0<|h|<h_0,
\]
where $\eta$ is as usual a $C^2$ cut-off function, such that $\eta\equiv 1$ on $B_{r_{i_0}}=B_{1/2}$. By choosing $\beta<1$ such that
\[
\frac{\beta+t+s\,p}{p}>1,
\]
we are then reduced to the previous subcase $\gamma_{i_0}>1$. The proof can then be concluded accordingly. We leave the technical details to the reader.
\begin{oss}[The number of iterations $i_0$]
All the estimates above crucially depends on the number of iterations $i_0\in\mathbb{N}\setminus\{0\}$. In particular, all the constants blow-up as $i_0$ goes to $\infty$.
It is thus useful to recall that if we set 
\[
\kappa=\kappa(t,s,p):=\frac{t+s\,p}{p-1},
\]
the sequence $\{\gamma_i\}_{i\in\mathbb{N}}$ has the following explicit expression
\[
\gamma_{i}=\frac{1}{p^i}\,s+\frac{t+s\,p}{p}\,\sum_{j=0}^{i-1}\,\frac{1}{p^j}=\frac{1}{p^i}\,s+\kappa\,\left(1-\frac{1}{p^i}\right),\qquad i\in\mathbb{N}.
\]
Then in the case $\boxed{t+s\,p\le (p-1)}$, the exponent $i_0$ is given by (recall that $s\le\tau<\kappa$)
\[
i_0=\min\left\{i\in\mathbb{N}\, :\, i>\frac{\ln(\kappa-s)-\ln(\kappa-\tau)}{\ln p}\right\},
\]
while in the case $\boxed{t+s\,p> (p-1)}$, this is given by
\[
i_0=\min\left\{i\in\mathbb{N}\, :\, i\ge \frac{\ln(\kappa-s)-\ln(\kappa-1)}{\ln p}\right\}.
\]
and we have
\[
\gamma_{i_0}=1\qquad \Longleftrightarrow\qquad \frac{\ln(\kappa-s)-\ln(\kappa-1)}{\ln p}\in\mathbb{N}\setminus\{0\}.
\]
\end{oss}

\subsection{Robust estimate for $s\nearrow 1$} We now reprove \eqref{leprosy} and \eqref{spiritualhealing}, this time for $s$ sufficiently close to $1$ and with an exact control on the constants. In other words, we want to prove the estimates \eqref{human} and \eqref{individual} claimed in Remark \ref{oss:constants}. We still denote by $u_R$ the scaled solution.
Let us thus fix $\ell_0>p$ and consider $0\le t\le s<1$ such that
\[
t+s\,(p+1)\ge \ell_0.
\]
Observe that sequence $\{\gamma_i\}_{i\in\mathbb{N}}$ defined in \eqref{exponents} is such that
\begin{equation}
\label{ince}
\gamma_1=\frac{t+s\,(p+1)}{p}\ge \frac{\ell_0}{p}>1,
\end{equation}
thus $i_0=1$ and we can conclude in one step, i.e. there is no need to iterate the estimate, exactly like in the usual case of the $p-$Laplacian. By using estimate \eqref{higher!!!} of Lemma \ref{lm:triebel} and \eqref{ince}, we immediately get 
\[
\|\nabla u_R\|_{L^p(B_{1/2})}\le \frac{C}{(\ell_0-p)^p}\,\left[(1-s)\,\mathcal{M}_{\gamma_{0}}+h_0^{-p\,\gamma_1}\,\|u_R\|^p_{L^p(B_1)}\right].
\]
This leads directly to (recall \eqref{step0} for $\mathcal{M}_{\gamma_0}$)
\[
\|\nabla u_R\|_{L^p(B_{r_1})}\le \frac{C}{(\ell_0-p)^p}\,(1-s)\,\mathcal{A}_1.
\]
By scaling back we get \eqref{human}. As for the fractional differentiability of $\nabla u$, we can reproduce the final step of the case $t+s\,p>(p-1)$ above ({\it case $\gamma_{i_0}>1$}). That is, we use \eqref{germano}, i.e.
\[
\sup_{0<|h|<h_0}\left\|\frac{\delta_h u_R}{|h|}\right\|^p_{L^p(B_{3/8})}\le C\,\|\nabla u_R\|^p_{L^p(B_{1/2})}.
\]
then  Proposition \ref{prop:derivatives} in the limit case $\gamma=1$ (with balls $B_{1/4}\Subset B_{1/2}$) and once more estimate \eqref{higher!!!bis} of Lemma \ref{lm:triebel}. We omit the details.

\subsection{A note on more general lower order terms}
\label{sec:general}
We spend some words on the case of the more general equation
\begin{equation}
\label{complete}
(-\Delta_{p,K})^s u=f+\Phi(u),
\end{equation}
where $\Phi:\mathbb{R}\to\mathbb{R}$ is a locally Lipschitz function.
This in particular embraces the case of eigenfunctions of $(-\Delta_p)^s$, corresponding to $f=0$, $K(z)=|z|^{N+s\,p}$ and $\Phi(t)=\lambda\,|t|^{p-2}\,t$ for some $\lambda>0$. This nonlinear and nonlocal eigenvalue problem has been first introduced in \cite{LL}. For more general nonlinearities $\Phi$, we address the reader to \cite{ILPS} for the existence theory.
\par
It is not difficult to see that Theorem \ref{teo:high} still holds for local weak solutions $u\in W^{s,p}_{loc}(\Omega)\cap \mathcal{Y}^{t,p}_s$ of \eqref{complete} such that
\[
u\in L^\infty_{loc}(\Omega).
\]
Indeed, the only difference with the proof of Theorem \ref{teo:high} is the presence of the additional term in the right-hand side of \eqref{regularity}
\[
\int |\Phi(u_h)-\Phi(u)|\, \left|\frac{\delta_h u}{|h|^{\gamma+t}}\right|\,\eta^p\,dx.
\]
This is of course a lower-order term, indeed it can be estimated as follows for $0<|h|<h_0<1$
\begin{equation}
\label{terminillo}
\begin{split}
\int |\Phi(u_h)-\Phi(u)|\, \left|\frac{\delta_h u}{|h|^{\gamma+t}}\right|\,\eta^p\,dx
&\le L\, \int_{B_\frac{R+r}{2}} \left|\frac{\delta_h u}{|h|^\frac{\gamma+t}{2}}\right|^2\,dx\le C\,L^\frac{p}{p-2}\,R^{N}+C\,\int_{B_R} \left|\frac{\delta_h u}{|h|^\gamma}\right|^p\,dx,
\end{split}
\end{equation}
where
\[
L=\sup_{\xi\in[-M,M]}|\Phi'(\xi)|\qquad \mbox{ and }\qquad M=\|u\|_{L^\infty(B_R)}.
\]
The last term in \eqref{terminillo} already appears in the right-hand side of \eqref{regularity}. Thus, the proof of Theorem \ref{teo:high} can be reproduced verbatim. Accordingly, estimates \eqref{screambloodygore}, \eqref{leprosy} and \eqref{spiritualhealing} still hold for bounded local weak solutions of \eqref{complete}, with the term $\mathcal{A}_R(u,f)$ defined in \eqref{AR} replaced by
\[
\mathcal{A}'_R(u,f):=\mathcal{A}_R(u,f)+L^\frac{p}{p-2}\, R^N,
\]
and $L$ is as above.
Remark \ref{oss:constants} about the quality of the relevant constants still applies.

\appendix 

\section{Proof of Proposition \ref{prop:yes}}
\label{sec:stein}

The proof is essentially the same as \cite[Chapter 5, Section V, Propositions 8' \& 9']{St}. The only difference is the use of the heat kernel, in place of the Poisson's one\footnote{In \cite{St} the space $\mathcal{B}^{\alpha,p}_{\infty}$ is denoted by $\Lambda_\alpha^{p,\infty}$.}.

\begin{proof}
We introduce the heat kernel
\[
\mathcal{K}_t(x)=\frac{1}{(4\,\pi\,t)^\frac{N}{2}}\,\exp\left(-\frac{|x|^2}{4\,t}\right),
\]
then we set
\[
\psi_t(x)=\mathcal{K}_t\ast \psi(x)=\frac{1}{(4\,\pi\,t)^\frac{N}{2}}\,\int_{\mathbb{R}^N} \exp\left(-\frac{|x-y|^2}{4\,t}\right)\,\psi(y)\,dy.
\]
Observe that by the semigroup property of the heat kernel we have
\[
\mathcal{K}_{t+s}=\mathcal{K}_{t}\ast \mathcal{K}_s,
\]
thus we get
\begin{equation}
\label{itsmagic!}
\frac{\partial}{\partial t}\nabla \psi_t=(\nabla\mathcal{K}_{t/2})\ast\left(\frac{\partial}{\partial t} \psi_{t/2}\right),
\end{equation}
where $\nabla$ denotes the gradient with respect to the $x$ variable.
In order to estimate the right-hand side of \eqref{itsmagic!} for $t>0$, we observe that\footnote{We have
\[
\frac{\partial}{\partial t}\mathcal{K}_t(x)=\mathcal{K}_t(x)\,\left[\frac{|x|^2}{4\,t^2}-\frac{N}{2\,t}\right].
\]
}
\[
\nabla \mathcal{K}_{t/2}(x)=-\frac{x}{t}\,\mathcal{K}_{t/2}(x),\qquad\|\nabla \mathcal{K}_t\|_{L^1(\mathbb{R}^N)}\le \frac{C}{\sqrt{t}},\qquad \int_{\mathbb{R}^N} \frac{\partial}{\partial t} \mathcal{K}_{t}(y)\,dy=0
\]
\[
\frac{\partial}{\partial t} \mathcal{K}_{t/2}(x)=\frac{\partial}{\partial t} \mathcal{K}_{t/2}(-x),\qquad \left|\frac{\partial}{\partial t}\,\mathcal{K}_{t/2}(x)\right|\le \frac{1}{2\,t}\,\mathcal{K}_{t/2}(x)\,\left|\frac{|x|^2}{t}-N\right|.
\]
Thus we get
\[
\frac{\partial}{\partial t} \psi_{t/2}(x)=\frac{1}{2}\,\int_{\mathbb{R}^N} \frac{\partial}{\partial t}\,\mathcal{K}_{t/2}(y)\,\Big[\psi(x+y)+\psi(x-y)-2\,\psi(x)\Big]\,dy.
\]
From this, by Minkowski inequality we obtain
\[
\begin{split}
\left\|\frac{\partial}{\partial t} \psi_{t/2}\right\|_{L^p(\mathbb{R}^N)}&\le \frac{1}{2}\,\int_{\mathbb{R}^N} \left|\frac{\partial}{\partial t}\mathcal{K}_{t/2}(y)\right|\,\Big\|\psi(\cdot+y)+\psi(\cdot-y)-2\,\psi\Big\|_{L^p(\mathbb{R}^N)}\,dy\\
&\le \frac{1}{2}\,[\psi]_{\mathcal{B}^{\alpha,p}_\infty(\mathbb{R}^N)}\,\int_{\mathbb{R}^N} \left|\frac{\partial}{\partial t}\mathcal{K}_{t/2}(y)\right|\,|y|^\alpha\,dy\\
&\le \frac{1}{4\,t}\,[\psi]_{\mathcal{B}^{\alpha,p}_\infty(\mathbb{R}^N)}\,\int_{\mathbb{R}^N} \mathcal{K}_{t/2}(y)\, \left|\frac{|y|^2}{t}-N\right|\,|y|^\alpha\,dy.
\end{split}
\]
With a simple change of variables, this gives
\[
\begin{split}
\left\|\frac{\partial}{\partial t} \psi_{t/2}\right\|_{L^p(\mathbb{R}^N)}&\le C\,[\psi]_{\mathcal{B}^{\alpha,p}_\infty(\mathbb{R}^N)}\,\left(\int_{\mathbb{R}^N} \mathcal{K}_1(z)\, \Big||z|^{2}-N\Big|\,|z|^\alpha\,dz\right)\,t^{\frac{\alpha}{2}-1}.
\end{split}
\]
Observe that for $1<\alpha<2$ 
\[
\int_{\mathbb{R}^N} \mathcal{K}_1(z)\, \Big||z|^{2}-N\Big|\,|z|^\alpha\,dz
\le \int_{\{|z|\le 1\}} \mathcal{K}_1(z)\,\Big||z|^{2}-N\Big|\,|z|\,dz+\int_{\{|z|>1\}} \mathcal{K}_1(z)\,\Big||z|^{2}-N\Big|\, |z|^2\,dz,
\]
and the last two terms are finite and depend only on $N$, so that in conclusion
\begin{equation}
\label{partialintime} 
\left\|\frac{\partial}{\partial t} \psi_{t/2}\right\|_{L^p(\mathbb{R}^N)}\le C\,[\psi]_{\mathcal{B}^{\alpha,p}_\infty(\mathbb{R}^N)}\,t^{\frac{\alpha}{2}-1},
\end{equation}
for some $C=C(N)>0$.
Then from \eqref{itsmagic!} and \eqref{partialintime} we obtain for every $t>0$
\begin{equation}
\label{11}
\left\|\frac{\partial}{\partial t}\nabla \psi_t\right\|_{L^p(\mathbb{R}^N)}\le \left\|\nabla \mathcal{K}_{t/2}\right\|_{L^1(\mathbb{R}^N)}\,\left\|\frac{\partial}{\partial t}\psi_{t/2}\right\|_{L^p(\mathbb{R}^N)}\le C\, [\psi]_{\mathcal{B}^{\alpha,p}_\infty(\mathbb{R}^N)}\,t^\frac{\alpha-3}{2}.
\end{equation}
We now integrate the previous inequality on the interval $(s,\tau)$, by Minkowski inequality again we get
\begin{equation}
\label{cauchy}
\begin{split}
\|\nabla \psi_\tau-\nabla \psi_s\|_{L^p(\mathbb{R}^N)}=\left\|\int_s^\tau \frac{\partial}{\partial t}\nabla \psi_t\,dt\right\|_{L^p(\mathbb{R}^N)}&\le \int_s^\tau\, \left\|\frac{\partial}{\partial t}\nabla \psi_t\right\|_{L^p(\mathbb{R}^N)}\,dt\\
&\le \frac{2\,C}{\alpha-1}\, [\psi]_{\mathcal{B}^{\alpha,p}_\infty(\mathbb{R}^N)}\,\left(\tau^\frac{\alpha-1}{2}-s^{\frac{\alpha-1}{2}}\right).
\end{split}
\end{equation}
Since $\alpha>1$ by assumption, this shows that $\{\nabla \psi_t\}_{0<t<1}$ is a Cauchy net in the complete space $L^p(\mathbb{R}^N)$. Thus there exists a sequence $\{t_k\}_{k\in\mathbb{N}}\subset (0,1)$ converging to $0$ as $k$ goes to $\infty$, such that $\{\nabla \psi_{t_k}\}_{k\in\mathbb{N}}$ converges strongly in $L^p$. The limit function is the distributional gradient of $\psi$. Finally, this shows that $\nabla \psi\in L^p(\mathbb{R}^N)$. Moreover, by taking the limit in \eqref{cauchy}, we get the estimate
\[
\|\nabla \psi\|_{L^p(\mathbb{R}^N)}\le \|\nabla \psi_1\|_{L^p(\mathbb{R}^N)}+\frac{2\,C}{\alpha-1}\, [\psi]_{\mathcal{B}^{\alpha,p}_\infty(\mathbb{R}^N)}\le C\,\|\psi\|_{L^p(\mathbb{R}^N)}+\frac{2\,C}{\alpha-1}\, [\psi]_{\mathcal{B}^{\alpha,p}_\infty(\mathbb{R}^N)},
\]
which is \eqref{02081980}.
\vskip.2cm\noindent
Once the existence of $\nabla \psi$ in $L^p$ is established, we can now prove \eqref{02082015}. We first need a decay estimate on the hessian $D^2 \psi_t$. For this, we observe that
\[
|D^2 \mathcal{K}_t(x)|\le \frac{\mathcal{K}_t(x)}{4\,t}\left[\frac{x\otimes x}{4\,t}-\mathrm{Id}_N\right]\qquad \mbox{ and }\qquad \|D^2 \mathcal{K}_t\|_{L^1(\mathbb{R}^N)}\le \frac{C}{t}.
\]
Then of course we have
\begin{equation}
\label{azzero}
\|D^2 \psi_t\|_{L^p(\mathbb{R}^N)}\le \frac{C}{t}\,\|\psi\|_{L^p(\mathbb{R}^N)}.
\end{equation}
Similarly as before, we can write
\[
\frac{\partial}{\partial t} D^2 \psi_t=\left(D^2\mathcal{K}_{t/2}\right)\ast\left(\frac{\partial}{\partial t} \psi_{t/2}\right),
\]
then for every $t>0$ we get
\[
\left\|\frac{\partial}{\partial t}D^2 \psi_t\right\|_{L^p(\mathbb{R}^N)}\le \left\|D^2 \mathcal{K}_{t/2}\right\|_{L^1(\mathbb{R}^N)}\,\left\|\frac{\partial}{\partial t}\psi_{t/2}\right\|_{L^p(\mathbb{R}^N)}\le C\, [\psi]_{\mathcal{B}^{\alpha,p}_\infty(\mathbb{R}^N)}\,t^\frac{\alpha-4}{2}.
\]
By integrating this estimate between $s$ and $T\gg s$, as above we get
\[
\|D^2 \psi_s\|_{L^p(\mathbb{R}^N)}\le \|D^2 \psi_T\|_{L^p(\mathbb{R}^N)}+\frac{C}{2-\alpha}\,[\psi]_{\mathcal{B}^{\alpha,p}_\infty(\mathbb{R}^N)}\,\left(s^\frac{\alpha-2}{2}-T^\frac{\alpha-2}{2}\right).
\]
By recalling that $\alpha<2$, using \eqref{azzero} and taking the limit as $T$ goes to $\infty$, we get the desired decay estimate 
\begin{equation}
\label{bloody}
\|D^2 \psi_s\|_{L^p(\mathbb{R}^N)}\le \frac{C}{2-\alpha}\,[\psi]_{\mathcal{B}^{\alpha,p}_\infty(\mathbb{R}^N)}\,s^\frac{\alpha-2}{2}.
\end{equation}
Let $h\in\mathbb{R}^N\setminus\{0\}$, by using that $\psi_t$ converges to $\psi$ as $t$ goes to $0$, we have
\begin{equation}
\label{calculus}
\delta_h\nabla \psi=\delta_h\nabla \psi_t-\int_0^t \frac{\partial}{\partial s} \left(\delta_h\nabla \psi_s\right)\,ds.
\end{equation}
By using the smoothness of $\psi_t$, we can write
\[
\delta_h\nabla \psi_t=\int_0^{|h|} \frac{d}{d\tau} \nabla \psi_t\left(x+\frac{h}{|h|}\,\tau\right)\,d\tau=\int_0^{|h|} D^2 \psi_t\left(x+\frac{h}{|h|}\,\tau\right)\,\frac{h}{|h|}\,d\tau,
\]
which implies
\[
\|\delta_h \nabla \psi_t\|_{L^p(\mathbb{R}^N)}\le \int_0^{|h|} \|D^2 \psi_t\|_{L^p(\mathbb{R}^N)}\,d\tau\le \frac{C}{2-\alpha}\,[\psi]_{\mathcal{B}^{\alpha,p}_\infty(\mathbb{R}^N)}\,|h|\,t^\frac{\alpha-2}{2},
\]
thanks to \eqref{bloody}. On the other hand, by triangle inequality and invariance of the $L^p$ norm by translations, we have
\[
\begin{split}
\left\|\frac{\partial}{\partial s} \left(\delta_h\nabla \psi_s\right)\right\|_{L^p(\mathbb{R}^N)}&\le 
2\,\left\|\frac{\partial}{\partial s} \nabla \psi_s\right\|_{L^p(\mathbb{R}^N)}\le C\, [\psi]_{\mathcal{B}^{\alpha,p}_\infty(\mathbb{R}^N)}\,s^\frac{\alpha-3}{2},
\end{split}
\]
where we also used \eqref{11}.
We can now use the two previous estimates in conjunction with \eqref{calculus}, so to get
\[
\begin{split}
\|\delta_h \nabla \psi\|_{L^p(\mathbb{R}^N)}&\le 
\|\delta_h \nabla \psi_t\|_{L^p(\mathbb{R}^N)}+\left\|\int_0^t \frac{\partial}{\partial s} \left(\delta_h\nabla \psi_s\right)\,ds\right\|_{L^p(\mathbb{R}^N)}\\
&\le C\,[\psi]_{\mathcal{B}^{\alpha,p}_\infty(\mathbb{R}^N)}\left[\frac{|h|}{2-\alpha}\,t^\frac{\alpha-2}{2}+\int_0^t s^\frac{\alpha-3}{2}\,ds\right]=C\,[\psi]_{\mathcal{B}^{\alpha,p}_\infty(\mathbb{R}^N)}\left[\frac{|h|}{2-\alpha}\,t^\frac{\alpha-2}{2}+\frac{2}{\alpha-1}\,t^\frac{\alpha-1}{2}\right].
\end{split}
\]
for some $C=C(N)>0$. The previous estimate holds for every $t>0$ and the right-hand side is minimal for $t=|h|^2/4$. With such a choice we thus get
\[
\left\|\frac{\delta_h \nabla \psi}{|h|^{\alpha-1}}\right\|_{L^p(\mathbb{R}^N)}\le \frac{C}{(2-\alpha)\,(\alpha-1)}\,[\psi]_{\mathcal{B}^{\alpha,p}_\infty(\mathbb{R}^N)},
\]
as desired.
\end{proof}

\section{Pointwise inequalities}
\label{sec:B}

For $p\ge 2$ we recall the definition of the functions $J_p:\mathbb{R}\to\mathbb{R}$ and $V_{p}:\mathbb{R}\to\mathbb{R}$
\[
J_p(t)=|t|^{p-2}\, t,\qquad \mbox{ and }\qquad V_p(t)=|t|^\frac{p-2}{2}\, t. 
\]
\begin{lm}
Let $p\ge 2$, for every $a,b\in\mathbb{R}$ we have
\begin{equation}
\label{monotone}
\Big(J_p(a)-J_p(b)\Big)\,(a-b)\ge (p-1)\,\left(\frac{2}{p}\right)^2\, \left|V_{p}(a)-V_{p}(b)\right|^2.
\end{equation}
\end{lm}
\begin{proof}
Since $J_p(a)-J_p(b)$ and $a-b$ share the same sign, we can assume without loss of generality that $a\ge b$.
If $a=b$ there is nothing to prove.
Let us assume that $a>b$, then we have
\[
\begin{split}
\Big(J_p(a)-J_p(b)\Big)\,(a-b)&=(p-1)\,\left(\int_b^a |t|^{p-2}\,dt\right)\,(a-b)\\
&\ge (p-1)\, \left(\int_a^b |t|^\frac{p-2}{2}\,dt\right)^2=(p-1)\,\left(\frac{2}{p}\right)^2\,|V_{p}(a)-V_{p}(b)|^2, 
\end{split}
\]
which concludes the proof.
\end{proof}
\begin{lm}
Let $p\ge 2$, for every $a,b\in\mathbb{R}$ we have
\begin{equation}
\label{lipschitz}
\Big|J_p(a)-J_p(b)\Big|\le 2\,\frac{p-1}{p}\, \left(|a|^\frac{p-2}{2}+|b|^\frac{p-2}{2}\right)\,\left|V_p(a)-V_p(b)\right|.
\end{equation}
\end{lm}
\begin{proof}
For $a=b$ there is nothing to prove. Let us consider the case $a\not=b$, without loss of generality, we can suppose that $a>b$. We set
\[
G(t)=|t|^\frac{p-2}{p}\,t,\qquad t\in\mathbb{R},
\]
by basic calculus we have
\[
\begin{split}
G\left(|a|^\frac{p-2}{2}\,a\right)-G\left(|b|^\frac{p-2}{2}\,b\right)
&\le \max\left\{G'\left(|a|^\frac{p-2}{2}\,a\right),\, G'\left(|b|^\frac{p-2}{2}\,b\right)\right\}\,\Big(V_p(a)-V_p(b)\Big).\\
\end{split}
\]
By observing that
\[
G\left(|t|^\frac{p-2}{2}\,t\right)=|t|^{p-2}\,t,
\]
we get the conclusion.
\end{proof}
\begin{lm}
Let $p\ge 2$, for every $a,b\in\mathbb{R}$ we have
\begin{equation}
\label{holder}
\left|V_p(a)-V_p(b)\right|^2\ge |a-b|^p.
\end{equation}
In particular, we also get
\begin{equation}
\label{down}
\Big(J_p(a)-J_p(b)\Big)\,(a-b)\ge (p-1)\,\left(\frac{2}{p}\right)^2\, |a-b|^p.
\end{equation}
\end{lm}
\begin{proof}
Observe that if $a=0$ or $b=0$, the result trivially holds. Thus let us suppose that $a\,b\not=0$ and observe that the function $H_p:\mathbb{R}\to\mathbb{R}$ defined by
\[
H_p(t)=|t|^\frac{2-p}{p}\,t,
\]
is $2/p-$H\"older continuous. More precisely, we have
\[
|H_p(t)-H_p(s)|\le |t-s|^\frac{2}{p},\qquad t,s\in\mathbb{R}.
\]
By observing that $H_p(V_p(t))=t$ and applying the previous with
\[
t=V_p(a)\qquad \mbox{ and }\qquad s=V_p(b),
\]
we get \eqref{holder}. The last inequality \eqref{down} follows by combining \eqref{monotone} and \eqref{holder}.
\end{proof}

\end{document}